\documentclass[12pt,a4paper]{amsart}
\usepackage{amsmath,amsthm,amsfonts,amssymb,mathrsfs}
\usepackage{graphicx}
\usepackage{color}

\newcommand\R{\mathbb{R}}
\newcommand\C{\mathbb{C}}

\newcommand\N{\mathcal{N}}
\newcommand\A{\mathcal{A}}
\newcommand\I{\mathcal{I}}
\newcommand\F{\mathcal{F}}

\def\bOmega{\overline \Omega}

\renewcommand\L{\mathcal{L}}
\renewcommand\Re{{\rm Re \,}}
\def\u{\bar{u}}
\def\v{\bar{v}}

\def\tv{\widetilde{v}}
\def\tu{\widetilde{u}}
\def\ve{\varepsilon}

\newtheorem{theorem}{Theorem}
\newtheorem{cor}[theorem]{Corollary}

\newtheorem{lemma}[theorem]{Lemma}
\newtheorem{prop}[theorem]{Proposition}

\numberwithin{equation}{section}
\numberwithin{theorem}{section}
\numberwithin{figure}{section}

\theoremstyle{remark}
\newtheorem{rem}[theorem]{Remark}

\theoremstyle{remark}

\theoremstyle{remark}
\newtheorem{assumption}[theorem]{Assumption}

\usepackage{delarray}
\usepackage{enumerate}

\begin{document}

\title[Instability of Turing patterns]{Instability of Turing patterns   in \\ reaction-diffusion-ODE systems}

\author[A. Marciniak-Czochra]{Anna Marciniak-Czochra}
\address[A. Marciniak-Czochra]{
Institute of Applied Mathematics,  Interdisciplinary Center for Scientific Computing (IWR) and BIOQUANT, University of Heidelberg, 69120 Heidelberg, Germany}
\email{anna.marciniak@iwr.uni-heidelberg.de}
\urladdr {http://www.biostruct.uni-hd.de/}

\author[G. Karch]{Grzegorz Karch}
\address[G. Karch]{
 Instytut Matematyczny, Uniwersytet Wroc\l awski,
 pl. Grunwaldzki 2/4, 50-384 Wroc\-\l aw, Poland}
\email{grzegorz.karch@math.uni.wroc.pl}
\urladdr {http://www.math.uni.wroc.pl/~karch}

\author[K. Suzuki]{Kanako Suzuki}
\address[K. Suzuki]{
College of Science, Ibaraki University,
2-1-1 Bunkyo, Mito 310-8512, Japan}
\email{kanako.suzuki.sci2@vc.ibaraki.ac.jp}

\date{\today}


\begin{abstract}
The aim of this paper is to contribute to the understanding of the pattern formation phenomenon in reaction-diffusion equations coupled with ordinary differential equations.  Such systems of equations arise, for example, from modeling of  interactions between cellular processes such as cell growth, differentiation or transformation and diffusing signaling factors. We focus on stability analysis of solutions of a prototype model consisting of a single reaction-diffusion equation coupled to an ordinary differential equation. We show that such systems are very different from classical reaction-diffusion models. They exhibit diffusion-driven instability (Turing instability)
 under a condition of autocatalysis of non-diffusing component. However, the same mechanism which destabilizes constant solutions of such models, destabilizes also all continuous spatially heterogeneous stationary solutions, and
consequently,  there exist no stable Turing patterns in such reaction-diffusion-ODE systems. 
We provide a rigorous result on the nonlinear instability, which involves 
the analysis of a continuous spectrum of a linear operator induced by the lack of diffusion in the destabilizing equation. 
These results are extended to discontinuous patterns for a class of nonlinearities.


\end{abstract}
\maketitle

\section{Introduction}

In this paper we focus on diffusion-driven instability (DDI) in systems of equations consisting of a single reaction-diffusion equation coupled with an ordinary differential equation system. Such systems  are important for systems biology applications; they arise for example in modeling of interactions between processes in cells and diffusing growth factors, such as in refs. \cite{Hock,Klika,MC03,MCK06, MCK08,Pham,USOO}. In some cases they can be obtained as homogenization limits of models describing coupling of cell-localized processes with cell-to-cell communication via diffusion in a cell assembly \cite{MCP,MC12}. Other examples are discussed e.g. in refs.  \cite{CTY06, E74,  MCNT13, MS09, WSW13}  and in the references therein. A detailed discussion of the DDI phenomena in the three-component systems with some diffusion coefficients equal to zero is found in the recent work \cite{Sakamoto12}.


 {\it Diffusion-driven instability}, also called the {\it Turing instability}, 
is a mechanism of {\it de novo} pattern formation, which has been often used to explain self-organization observed in nature.
DDI is a bifurcation that arises in a reaction-diffusion system, when there exists a spatially homogeneous stationary solution which is asymptotically stable  with respect to spatially homogeneous perturbations but  unstable to spatially heterogeneous perturbations.  
Models with DDI describe a process of a destabilization of  stationary spatially homogeneous 
steady states and 
evolution of the system  towards  spatially heterogeneous steady states. DDI has  inspired 
a vast number of mathematical models since the seminal paper of  Turing \cite{Turing}, providing explanations of symmetry breaking and {\it de novo} pattern formation,  shapes of animal coat markings,  and oscillating chemical reactions.
We refer the reader to the monographs by Murray \cite{MurrayI, MurrayII} and to the review
article \cite{kanako} %
for references on DDI  in the two component reaction-diffusion systems and to the paper
\cite{SMM} in the several component systems. %

%

However, in many applications there are components which are localized in space, which leads to  systems of  ordinary differential equations coupled with  reaction-diffusion equations. 
Our main goal is to clarify in what manner such models
are different from the classical Turing-type models and to demonstrate that the spatial structure of the pattern emerging via DDI cannot be determined based on linear stability analysis.

To understand the role of non-diffusive components in the pattern formation process, we focus on systems involving a single reaction-diffusion equation coupled to ODEs.  It is an interesting case, since a scalar reaction-diffusion equation cannot exhibit stable spatially heterogenous patterns \cite{CaHo} and hence in such models it is the ODE component that yields the patterning process. As shown in ref. \cite{MKS12}, it may happen that there exist no stable stationary patterns  and the emerging spatially heterogeneous structures are of a dynamical nature.
 In numerical simulations of such models,  solutions having the form of  unbounded periodic or irregular spikes have been observed \cite{MC13H}.

Thus, the aim of this  work is to investigate to which extent the results obtained in \cite{MKS12}, concerning the instability of all stationary structures, apply to a general class of reaction-diffusion-ODE models with a single diffusion operator.

 We focus on the following two-equation system 
\begin{align}
u_t  &=   f(u,v),&     \text{for}\quad &x\in\overline{\Omega}, \;\quad t>0, && \label{eq1}\\
v_t  &=    \Delta v+g(u,v)&  \text{for}\quad &x\in \Omega, \quad t>0&&\label{eq2}
\end{align}
in a bounded domain $\Omega\subset \R^N$ for $N \geq 1$, with a $C^2$-boundary $\partial\Omega$, 
supplemented with the Neumann boundary condition 
\begin{equation}\label{Neumann}
\partial_{\nu}v  =  0  \qquad \text{for}\quad x\in \partial\Omega, \quad t>0,
\end{equation}
where $\partial_\nu = \frac{\partial}{\partial \nu}$ and $\nu$ denotes the unit
outer normal vector to $\partial \Omega$, and with initial data
\begin{eqnarray}\label{ini}
&&u(x,0)  =   u_{0}(x),\qquad v(x,0)  =  v_{0}(x).
\end{eqnarray}
The nonlinearities 
 $f=f(u,v)$ and $g=g(u,v)$ are arbitrary $C^3$-functions.
Notice  that equation \eqref{eq2} may contain an arbitrary diffusion coefficient which, however,
can be rescaled and assumed to be equal to one.



In this paper we investigate  stability properties of  stationary solutions of the problem  \eqref{eq1}-\eqref{Neumann}.
Our main results are Theorems \ref{thm:non-const}
and \ref{thm:weak}
which assert that, under a natural assumption satisfied by a wide variety of systems, stationary solutions are unstable.
We call this assumption the {\it autocatalysis condition} 
(see Theorem \ref{thm:non-const})
following its physical motivation in the model. 
We show in Section \ref{sec:examp}
that this condition  is satisfied {\it for all stationary solutions} of a wide class of systems 
from mathematical biology.
 Our results are different in continuous and discontinuous stationary solutions. In the latter case,  additional assumptions on the structure of nonlinearities are required.

As a complementary result to the instability theorems, we prove
Theorem \ref{cor3} which states  that each non-constant regular stationary solution
intersecting (in a sense to be defined)  constant steady states with the DDI property, has to satisfy the 
autocatalysis condition. It is a classical idea by Turing that stable patterns appear around the
 constant steady state in systems of reaction-diffusion equations with DDI. Mathematical results on stability of such patterns can be found, {\it e.g.}, in refs.
\cite{IWW04, Wei08, WW07, WW08, WW14} and in the references therein.
In the current work, 
combining Theorems \ref{thm:non-const} and \ref{cor3},
we show that {\it this is not  the case} in
the reaction-diffusion-ODE problems \eqref{eq1}-\eqref{ini}. 
In other words, {\it the same mechanism which
destabilizes constant solutions of such models, destabilizes also 
non-constant stationary  solutions}, a behavior that does not fit the usual paradigm of the reaction-diffusion-type equations. See Remark~\ref{rem:Turing} for more details. 

Mathematically, in the proof of our main result we need to consider  a nonempty continuous spectrum of the linearized  operator.  This seems to be a novelty in the study of reaction-diffusion equations, and is caused by the absence of diffusion in one of the equations. 
In Section \ref{sec:instab}, 
we provide a rigorous proof of the nonlinear instability of steady states
by using ideas from fluid dynamics equations.

The paper is organized as follows. In Section \ref{sec:main}, we state the main results. Section \ref{sec:examp} provides relevant mathematical biology-related examples of reaction-diffusion-ODE systems. Proofs are postponed to Sections \ref{sec:instab} and \ref{sec:const:non}.   
Section \ref{sec:instab} is devoted to showing instability of the  stationary solutions under  the autocatalysis  condition.
A proof of the instability of discontinuous solutions requires additional conditions 
on the model nonlinearities. In Section \ref{sec:const:non},
the continuous  stationary solutions  are characterized and it is shown that 
the autocatalysis condition is satisfied in the class of reaction-diffusion-ODE problems \eqref{eq1}-\eqref{ini} exhibiting DDI. Appendix contains additional information on the model of early carcinogenesis which was the main motivation for our research.

\section{Results and comments}\label{sec:main}

First we  formulate  a condition which leads to instability of regular stationary solutions of the problem  \eqref{eq1}-\eqref{ini}.
Then, we show that it is the necessary condition for DDI in reaction-diffusion-ODE systems. 
Finally, we extend the instability results to a class of discontinuous stationary solutions satisfying additional assumptions.

\subsection{Instability of regular steady states}\label{sec2}

First, we focus on {\it regular stationary solutions} $(U,V)$ of problem \eqref{eq1}-\eqref{Neumann}. For this, we assume that there exists a solution  (not necessarily unique) of the equation 
$f\big(U(x),V(x)\big)=0$ that is given by the relation 
 $U(x)=k(V(x))$  for all $x\in \Omega$
with a $C^1$-function $k=k(V)$.
Then,  every  regular solution $(U,V)$ of the  boundary value problem
\begin{align}
&f(U,V)=0&     \text{for}\quad &x\in\overline{\Omega},  && \label{seq1}\\
 &\Delta V+g(U,V)=0&  \text{for}\quad &x\in \Omega, &&\label{seq2}\\
&\partial_{\nu}V  =  0&  \text{for}\quad &x\in \partial\Omega &&\label{sNeumann}
\end{align}
satisfies the elliptic problem  
\begin{align}
 &\Delta V+h(V)=0&  \text{for}\quad &x\in \Omega, &&\label{s:h}\\
& \partial_{\nu}V  =  0 & \text{for}\quad &x\in \partial\Omega,&&\label{s:h:N} 
\end{align}
where 
\begin{equation}\label{hgUk}
h(V)=g\big(k(V),V\big) \qquad \text{and}\qquad U(x)=k(V(x)).
\end{equation}

We show that all regular  stationary solutions of the problem \eqref{eq1}-\eqref{ini} are unstable under a simple assumption on the first equation.

\begin{theorem}[Instability of regular solutions]\label{thm:non-const}
Let $(U, V)$ be a  regular  solution of the problem \eqref{seq1}-\eqref{sNeumann}
satisfying
the following ``autocatalysis condition'':
\begin{equation}\label{as:auto}
 f_u \big(U(x_0), V(x_0)\big) > 0 \quad \text{for some}\quad  x_0 \in {\Omega}.
\end{equation}
Then, $(U, V)$ is an unstable solution of the initial-boundary value problem \eqref{eq1}-\eqref{ini}.
\end{theorem}

Inequality  \eqref{as:auto} can be interpreted as
{\it autocatalysis} in the dynamics of $u$
at the steady state $(U,V)$ at some point $\ x_0 \in {\Omega}$. 
Stability of the stationary solution is understood in the Lyapunov sense.
Moreover, we prove nonlinear instability of the stationary solutions of the problem \eqref{eq1}-\eqref{Neumann}
and not only their linear instability, {\it i.e.} the instability of zero solution of the corresponding linearized problem, see Section \ref{sec:instab} for more explanations.

Each constant solution $(\u,\v)\in \R^2$ of the problem  \eqref{seq1}-\eqref{sNeumann} is a particular case of a regular solution.
Thus, Theorem \ref{thm:non-const} provides a simple criterion for 
the diffusion-driven instability (DDI)
 of $(\u,\v)$.

\begin{cor}\label{cor:instab}
If a constant solution $(\u, \v)$ of the problem \eqref{eq1}-\eqref{ini} 
{\rm (}namely, $f(\u,\v)=$ and $g(\u,\v)=0${\rm )}
satisfies the inequalities 
\begin{equation}\label{DDI:cor}
f_u(\u,\v)> 0, \qquad 
f_u(\u,\v)+g_v(\u,\v)< 0, \qquad 
\det
\left(
\begin{array}{cc}
f_u(\u,\v)&f_v(\u,\v)\\
g_u(\u,\v)&g_v(\u,\v)
\end{array}
\right)> 0, 
\end{equation}
then it  has the  DDI property.
\end{cor}

This corollary follows directly from Theorem \ref{thm:non-const}, because the second and the third  inequality in \eqref{DDI:cor} imply 
that  $(\u, \v)$ is stable under homogeneous perturbations; see Remark~\ref{rem:kin} for more details.


\subsection{Sufficient conditions for autocatalysis.}

Next, we show that DDI in the problem \eqref{eq1}-\eqref{ini} implies the autocatalysis condition \eqref{as:auto}.

We consider only a {\it non-degenerate}
 constant stationary solution $(\u,\v)$ of the reaction-diffusion-ODE system \eqref{eq1}-\eqref{Neumann}. Hence, in the remainder of this work we make the following  assumption.

\begin{assumption}[Non-degeneracy of the stationary solutions]\label{ass}
Let all stationary solutions, i.e. vectors $(\u,\v)\in\R^2$ such that
$
f(\u,\v)=0 
$
and 
$g(\u,\v)=0$, satisfy
\begin{equation}\label{non-deg}
f_u(\u,\v)+g_v(\u,\v)\neq 0, \quad 
\det
\left(
\begin{array}{cc}
f_u(\u,\v)&f_v(\u,\v)\\
g_u(\u,\v)&g_v(\u,\v)
\end{array}
\right)\neq 0, \quad \mbox{and} \quad f_u(\u,\v)\neq 0.
\end{equation}
\end{assumption}

\begin{rem}\label{rem:kin}
Let us note that the 
 first two conditions in  \eqref{non-deg} allow us to study the asymptotic 
stability of $(\u,\v)$ treated  as a solution of the corresponding system of ordinary differential equations
\begin{equation}\label{k1}
\frac{du}{dt}=f(u,v), \qquad \frac{dv}{dt}=g(u,v),
\end{equation}
by analyzing eigenvalues of the corresponding linearization matrix.
Indeed,
the conditions for linearized stability read
\begin{itemize}
\item[1.] If
\begin{equation}\label{k:stab}
f_u (\u, \v) + g_v (\u, \v) < 0 \qquad \mbox{and} \qquad 
\det
\left(
\begin{array}{cc}
f_u(\u,\v)&f_v(\u,\v)\\
g_u(\u,\v)&g_v(\u,\v)
\end{array}
\right) > 0,
\end{equation}
then the Jacobi matrix
\begin{equation}\label{Jacobi}
\left(
\begin{array}{cc}
f_u(\u,\v)&f_v(\u,\v)\\
g_u(\u,\v)&g_v(\u,\v)
\end{array}
\right)
\end{equation}
has all eigenvalues with negative real parts, and hence
$(\u, \v)$ is an asymptotically stable solution of  system \eqref{k1}.
\item[2.] On the other hand, if
\begin{equation}\label{k:instab}
\mbox{either}\qquad 
f_u (\u, \v) + g_v (\u, \v) > 0 \qquad \mbox{or} \qquad 
\det
\left(
\begin{array}{cc}
f_u(\u,\v)&f_v(\u,\v)\\
g_u(\u,\v)&g_v(\u,\v)
\end{array}
\right) < 0,
\end{equation}
then the linearization matrix  \eqref{Jacobi} has an eigenvalue 
with a positive real part, and consequently,  
the pair $(\u, \v)$ is an unstable solution of  \eqref{k1}.
\end{itemize}
\end{rem}

Now, we state a simple but   fundamental property of the stationary solutions of the problem \eqref{eq1}-\eqref{ini}.

\begin{prop}\label{prop1}
Assume that $(U,V)$ is a  non-constant regular solution of  the stationary problem
\eqref{seq1}-\eqref{sNeumann}.
Then, there exists $x_0\in\bOmega$, such that the
vector
$
(\u,\v)\equiv \big(U(x_0),  V(x_0)\big)
$
is a constant solution of the problem
\eqref{seq1}-\eqref{sNeumann}.
\end{prop}


To prove Proposition \ref{prop1},  it suffices to 
 integrate  equation  \eqref{seq2}  over $\Omega$ and  to use the Neumann boundary condition \eqref{sNeumann} to obtain
$
\int_{\Omega} g\big(U(x),V(x)\big)\;dx =0.
$
Hence, there exists $x_0\in\bOmega$ such that $g\big(U(x_0),V(x_0)\big)=0$, because $U$ and $V$ are continuous.
Thus, by equation \eqref{seq1}, it also holds
$
f\big(U(x_0),V(x_0)\big)=0.
$

In the case described  by Proposition \ref{prop1}, we say that 
{\it a non-constant solution $(U,V)$ intersects a constant solution $(\u,\v)$}.
Now, we prove an important  property of  the constant solutions that are intersected by non-constant regular solutions.

\begin{prop}\label{thm:ineq}
 Let $\big(U(x),V(x)\big)$ be a regular non-constant stationary solution of problem \eqref{eq1}--\eqref{Neumann}
and assume that all constant stationary solutions that are intersected by 
$(U,V)$ are non-degenerate, {\it i.e.}~relations  \eqref{non-deg} are satisfied. 
Then, at least at one of those constant solutions,
denoted here by $(\u,\v)$, the following inequality holds
\begin{equation}\label{ineq}
\frac{1}{f_u(\u,\v)}\det 
\left(
\begin{array}{cc}
f_u(\u,\v)&f_v(\u,\v)\\
g_u(\u,\v)&g_v(\u,\v)\\
\end{array}
\right)
>0.
\end{equation}
\end{prop}


The proof of Proposition \ref{thm:ineq} 
is based on the properties of the solutions of the elliptic Neumann problem 
\eqref{s:h}--\eqref{s:h:N}
(see Theorem \ref{lem:N}, below), which we prove in  Section \ref{sec:const:non}.

\begin{rem}\label{rem0}
Every  non-degenerate constant solution $(\u,\v)$
of the problem \eqref{eq1}-\eqref{ini} satisfying inequality \eqref{ineq} is
{\it  unstable}. 
If both factors on the left-hand side of inequality \eqref{ineq}
are positive, then, in particular, the autocatalysis  
condition $f_u(\u,\v)>0$ is satisfied.
Hence, the constant solution $(\u,\v)$
is an unstable solution of the reaction-diffusion-ODE system \eqref{eq1}-\eqref{ini} by Theorem \ref{thm:non-const}.
On the other hand, 
  if both factors on the left-hand side of inequality \eqref{ineq} are
 negative, %
then, in particular,  the determinant in inequality \eqref{ineq} is negative and
the constant vector $(\u,\v)$
is an   unstable solution of the corresponding  kinetic system \eqref{k1},
see the alternative \eqref{k:instab} in Remark \ref{rem:kin}.
\end{rem}

\begin{rem}
It is worth to emphasize the following particular case of the phenomenon described in 
Remark \ref{rem0}, because we shall encounter  it   in our examples, further on.
Suppose that the problem \eqref{eq1}-\eqref{ini} has a non-constant regular
stationary solution  $(U,V)$ 
intersecting {\it only one} constant and non-degenerate
steady state  $(\u,\v)$ which is asymptotically stable as a solution of the kinetic system \eqref {k1}.
In such case, 
inequality \eqref{ineq} together with the second inequality in \eqref{k:stab} directly imply 
 the 
autocatalysis  condition
$
f_u(\u,\v)>0.
$
Thus, by Theorem \ref{thm:non-const},
 $(\u,\v)$ is an unstable solution of the
reaction-diffusion-ODE problem \eqref{eq1}-\eqref{ini}, {\it i.e.} the constant steady state $(\u,\v)$ has the DDI property.
Below, in Theorem~\ref{cor3}, we show that 
the non-constant stationary solution 
$(U,V)$ 
also satisfies the autocatalysis condition 
\eqref{as:auto}, and hence, it is unstable.
\end{rem}

%
%
%

Now, we are in the position to show that the autocatalysis condition
\eqref{as:auto} has to be satisfied in reaction-diffusion-ODE systems 
\eqref{eq1}--\eqref{Neumann}
with non-constant regular stationary solutions which intersect
 constant steady states with the DDI property.

\begin{theorem}\label{cor3}
Let  $(U,V)$ be a non-constant regular stationary solution  of problem \eqref{eq1}-\eqref{ini}.
Denote by  $(\u,\v)$ 
 a non-degenerate constant solution  
 which intersects $(U,V)$,
 and  satisfies inequality \eqref{ineq}.
 Assume that  $(\u,\v)$ is an asymptotically  stable solution of  the kinetic system \eqref{k1}.
Then, 
there exists $x_0\in\Omega$ such that 
$$
f_u\big(U(x_0),V(x_0)\big)=f_u(\u,\v)>0.
$$
\end{theorem}

The following remark emphasizes importance 
 of  the above results.

\begin{rem}\label{rem:Turing}
The instability results from Theorem \ref{thm:non-const} and Corollary \ref{cor:instab} combined with Theorem \ref{cor3}
can be summarized in the following way.
This is a classical idea that, in a system of reaction-diffusion equations with a constant solution having
the DDI property,  one expects
 stable patterns to appear around that constant steady state. 
Such stationary solutions are called the {\it Turing patterns}.
For  the initial-boundary value problem for a reaction-diffusion-ODE system with a single diffusion equation
\eqref{eq1}-\eqref{Neumann},
such stationary solutions can be constructed in the case of several models of interest (see Section~\ref{sec:examp}).
However,  the same mechanism that
destabilizes constant solutions of such models, also destabilizes the non-constant solutions. In other words, {\it all Turing patterns in  the reaction-diffusion-ODE
problems  \eqref{eq1}-\eqref{Neumann} are unstable.}
\end{rem}

\subsection{Instability of non-regular steady states}\label{sec2.3}
The initial-boundary value problem \eqref{eq1}-\eqref{ini} may also have non-regular steady states in the case when the equation $f(U,V)=0$ is not uniquely solvable. Choosing different branches of solutions of the equation $f\big(U(x),V(x)\big)=0$,
we obtain the relation $ U(x)=k\big(V(x)\big)$ with a discontinuous, piecewise $C^1$-function $k$.
Here, we recall that a pair
 $\big(U,V\big)\in L^\infty(\Omega)\times W^{1,2}(\Omega)$ is 
a {\it weak solution} of problem   \eqref{seq1}-\eqref{sNeumann} 
if the equation $f\big(U(x),V(x)\big)=0$ is satisfied for almost all $x\in \Omega$ and if
\begin{equation*}
-\int_\Omega \nabla V(x)\cdot \nabla \varphi(x)\;dx +\int_\Omega g\big(U(x),V(x)\big)\varphi(x)\;dx=0
\end{equation*}
holds for all test functions $\varphi\in W^{1,2}(\Omega)$.

In this work, we do not prove the existence of such discontinuous solutions
and we refer the reader to 
classical works \cite{ATW88,MTH80,Sakamoto}
as well as 
to our recent paper \cite[Thm.~2.9]{MKS12} for information about how to construct such solutions 
to one dimensional problems using phase portrait analysis.
Our goal is to formulate  a counterpart of the autocatalysis condition  \eqref{as:auto},  
which leads to instability of the weak (including  discontinuous) 
stationary solutions.

\begin{theorem}\label{thm:weak}
Assume that $(U, V)$ is a  weak   bounded solution of the problem \eqref{seq1}-\eqref{sNeumann} satisfying the following counterpart of the autocatalysis 
condition 
\begin{equation}\label{auto:weak}
\lambda_0\leq  f_u(U(x),V(x))\leq \Lambda_0
\qquad \text{for almost  all} \quad x\in \overline\Omega,
\end{equation}
for some constants $0<\lambda_0\leq \Lambda_0<\infty$. Suppose, moreover, that 
there exists $x_0\in\Omega$ such that $f_u(U,V)$ is continuous in a neighborhood of $x_0$.
Then, $(U, V)$ is an unstable solution of the initial-boundary value problem \eqref{eq1}-\eqref{ini}.
\end{theorem}

We prove Theorem \ref{thm:weak} in Subsection \ref{subsec:instab} by applying ideas developed
for the analysis of the Euler equation and other fluid dynamics models.
In that approach, it suffices to show that  the spectrum of the linearized operator
\begin{equation*}\label{lin:0}
\L
\left(
\begin{array}{c}
\tu\\
\tv
\end{array}
\right)
\equiv
\left(
\begin{array}{c}
0 \\
\Delta \tv
\end{array}
\right)
+
\left(
\begin{array}{cc}
f_u(U, V)&f_v(U, V)\\
g_u(U, V)&g_v(U, V)
\end{array}
\right)
\left(
\begin{array}{c}
\tu\\
\tv
\end{array}
\right)
\end{equation*}
with the Neumann boundary condition
$
\partial_\nu \tv=0,
$ 
has so-called {\it spectral gap}, namely, there exists a subset 
 of the spectrum $\sigma(\L)$, which 
has a positive real part, 
 separated from zero. Here, we prove  that $\sigma(\L)\subset \C$
consists of the set $\{f_u(U(x),V(x))\, :\, x\in \overline \Omega \}$ and of 
isolated eigenvalues of $\L$, see  Section \ref{sec:instab}
and, in particular, Fig.~\ref{fig1}
 for more detail.
One should emphasize that the instability of steady states from 
Theorems~\ref{thm:non-const} and \ref{thm:weak}
is caused not by an eigenvalue with a positive real part, but rather by  positive 
numbers from the set ${\rm Range}\, f_u(U,V)$ which is
contained in the
continuous spectrum of the operator $(\L, D(\L))$, see Theorem  \ref{thm:spec:L} below for more details.

In fact, in the case of particular nonlinearities, we do not need to assume that  condition \eqref{auto:weak} 
holds true for almost  all $x\in\Omega$.
Indeed, 
 if $f(0,v)=0$, one may have stationary solutions $U=U(x)$ 
such that $U(x)=0$ on a subset of $\Omega$  and $U(x)>0$ on a complement.
 Such stationary solutions can be, for example, constructed for
  the carcinogenezis model \eqref{eqc1}-\eqref{eqc3} presented below (see \cite{MKS12}), and for several other one-dimensional equations discussed in ref.
\cite{MTH80}.
In the following corollary, we show instability of the
discontinuous stationary solutions, under the autocatalysis condition
 only for $x\in \Omega$ such that $U(x)\neq 0$.

\begin{cor}[Instability of weak solutions]\label{cor:weak}
Assume that the nonlinear term  in the equation \eqref{eq1} satisfies  $f(0,v)=0$ for all $v\in \R$.
Suppose that  $(U, V)$ is a  weak  bounded  solution of the problem \eqref{seq1}-\eqref{sNeumann}
with the following property: There exist  constants $0<\lambda_0<\Lambda_0<\infty$ 
such that 
\begin{equation}\label{as:auto:dis}
 \lambda_0\leq f_u \big(U(x), V(x)\big) \leq \Lambda_0 \qquad
\text{for almost all}\quad  x\in\Omega,\quad \text{where}\quad  U(x)\neq 0.
\end{equation}
 Moreover, suppose that there exists $x_0\in \Omega$ such that $U(x_0)\neq 0$ and 
the functions 
$U=U(x)$ and $f_u(U,V)$ are continuous in the neighborhood of $x_0$.
Then, $(U, V)$ is an unstable solution of the initial-boundary value problem \eqref{eq1}-\eqref{ini}.
\end{cor}

\begin{rem}
A typical nonlinearity satisfying the assumptions of Corollary  \ref{cor:weak}
has the form $\mathrm{f(u,v)=r(u,v)u}$. It can be found 
in the models, where the unknown variable
$u$ evolves  according to the Malthusian law with a density dependent growth rate $r$.
\end{rem}

 We defer the proofs of Theorems \ref{thm:non-const}
and \ref{thm:weak} 
as well as of 
Corollary \ref{cor:weak}
to Subsection~\ref{subsec:instab}.
Theorem \ref{cor3} is somewhat independent of  Theorems \ref{thm:non-const}
and \ref{thm:weak} and it is proven in Section \ref{sec:const:non}.

\section{Model examples}\label{sec:examp}

In this section, our results  are illustrated by applying them to some
models  from mathematical biology.

\subsection{Gray-Scott model.}

First we consider a reaction-diffusion-ODE model with nonlinearities as in the celebrated 
Gray-Scott system describing pattern formation in chemical reactions ~\cite{GrayScott}. The system with a non-diffusing activator takes the form
\begin{align}
\label{eq1s} u_t  &=  - (B+k) u + u^2 v&     \text{for}\quad &x\in\overline{\Omega}, \; t>0, && \\
\label{eq2s} v_t  &=   \Delta v -u^2 v+B(1-v)&  \text{for}\quad &x\in \Omega, \; t>0,&&
\end{align}
with the zero-flux boundary condition for $v$
and with nonnegative  initial conditions. The constants $B$ and $k$ are assumed to be positive.
The system exhibits the instability phenomenon described  in
Section  \ref{sec:main}.

Here,  every regular positive stationary solution $(U,V)$ of the Neumann boundary-initial value problem for equations \eqref{eq1s}-\eqref{eq2s} has to satisfy  the relation $U = (B+k)/ V$, hence,
\begin{align}
 &\Delta V-B V-\frac{(B+k)^2}{V}+B=0&  \text{for}\quad &x\in \Omega, &&\label{S:s}\\
& \partial_{\nu}V  =  0 & \text{for}\quad &x\in \partial\Omega.&&\label{S:s:N} 
\end{align}
All continuous positive solutions of such boundary value problem in one dimensional case
have been constructed in our recent paper \cite[Sec.~5]{MKS12}. 
A construction of discontinuous stationary solutions of the reaction-diffusion-ODE problem for \eqref{eq1s}-\eqref{eq2s} can be also found in \cite[Thm.~2.9]{MKS12}.

Instability 
results in Theorems \ref{thm:non-const} and \ref{cor3} 
imply that {\it all stationary solutions}  (constant, regular as well as discontinuous) 
of the reaction-diffusion-ODE problem \eqref{eq1s}-\eqref{eq2s} are unstable under heterogeneous perturbations. 
For the proof, 
it suffices to notice that 
the autocatalysis assumptions \eqref{as:auto} and \eqref{auto:weak}
are satisfied, since, for $U=(B+k)/V$,  the function $f_u\big(U(x),V(x)\big)$ is independent of $x$ and satisfies
$$
 f_u \big(U(x), V(x)\big) = -(B+k) + 2 U(x)V(x) = B+k > 0
\qquad \text{for all} \quad x\in \Omega.
$$

\subsection{Model of early carcinogenesis}\label{sub:carc}
The main motivation for the  research reported in this work 
has been the study
of the 
 reaction-diffusion system  of three ordinary/partial differential equations 
modeling   the diffusion-regulated 
growth of a cell population of the following form
\begin{align}
u_t&=\Big(\frac{a v}{u+v} -d_c\Big) u  &\text{for}\ x\in \overline\Omega, \; t>0, \label{eqc1}\\
v_t &=-d_b v +u^2 w -d v   &\text{for}\ x\in \overline\Omega, \; t>0, \label{eqc2}\\
w_t &= D \Delta w -d_g w -u^2 w +d v +\kappa_0  &\text{for}\ x\in \Omega, \; t>0, \label{eqc3}
\end{align}
supplemented with zero-flux boundary conditions for the function $w$
and with nonnegative initial conditions,  \cite{MKS12}.
Here, the letters $a , d_c,d_b, d_g, d, D, \kappa_0 $  denote positive constants.


The theory developed in this paper applies to a reduced two-equation version of  the model \eqref{eqc1}-\eqref{eqc3}, obtained using a  quasi-steady state approximation of the dynamics of $v$.  Applying  the quasi-steady state approximation in equation \eqref{eqc2} (i.e., setting  $v_t\equiv 0$), we obtain the relation
$
v={u^2 w}/({d_b+d}),
$
which after substituting 
into  the remaining equations \eqref{eqc1} and \eqref{eqc3} yields the 
initial-boundary value problem for the following reaction-diffusion-ODE system
\begin{align}
u_t&=\Big(\frac{a uw}{d_b+d+uw} -d_c\Big) u &\text{for}\ x\in \overline\Omega, \; t>0, \label{eqn1}\\
w_t &= D\Delta w -d_g w -\frac{d_b}{d_b+d} u^2 w  +\kappa_0 & \text{for}\ x\in \Omega, \; t>0. \label{eqn2}
\end{align}
A rigorous derivation of the two equation model \eqref{eqn1}-\eqref{eqn2} from the model \eqref{eqc1}-\eqref{eqc3}
as well as other properties of the solutions to \eqref{eqn1}-\eqref{eqn2} are presented in Appendix A of this work. Moreover, numerical simulations suggest that  the two-equation model exhibits qualitatively the same dynamics as 
system
\eqref{eqc1}-\eqref{eqc3}.

The autocatalysis assumptions \eqref{as:auto} and \eqref{auto:weak}
are
satisfied by  simple calculations, similar to those in the previous example (see \cite{MKS12} for more details).
As a consequence, {\it all nonnegative stationary solutions} of the system \eqref{eqn1}-\eqref{eqn2}
(regular and non-regular)
are  unstable due to  Theorems~\ref{thm:non-const} and \ref{thm:weak}. 
This corresponds to our results on the  three-equation model \eqref{eqc1}-\eqref{eqc3} proved in ref. \cite{MKS12}.
 
Stability analysis of the space homogeneous solutions of the two equation model \eqref{eqn1}-\eqref{eqn2}  is reported in Appendix B.
In particular, by Remark \ref{rem0},
 constant steady states of  \eqref{eqn1}-\eqref{eqn2} are either unstable solutions of the corresponding kinetic system or they have the DDI property.


\subsection{Model of glioma invasion}\label{sub:glioma}

Our results can be also applied to the 
 {\it ``go-or-grow'' model}  introduced in ref. \cite{Pham}
to investigate the dynamics of a
population of glioma cells switching between a migratory and a proliferating phenotype in dependence on the local cell density. The model consists of two reaction-diffusion equations 
\begin{align}
u_t&=- \mu \Big(\Gamma(u+v)u-\big(1-\Gamma(u+v)\big)v\Big)
+ru\big(1-(u+v)\big)\label{gog1}
\\
v_t&= \Delta v +\mu \Big(\Gamma(u+v)u-\big(1-\Gamma(u+v)\big)v\Big),\label{gog2}
\end{align}
where tumor cells are decomposed into two sub-populations: a migrating population with density $v(x,t)$ and a proliferating population with  density 
$u(x,t)$ (Caution: we changed the notation from \cite{Pham}, where $\rho_1=v$ and $\rho_2=u$).
In this model, the constant $\mu>0$ is the rate at which cells change their phenotype and the constant $r\geq 0$ is the proliferation rate. The function $\Gamma=\Gamma(\rho)$ has the following explicit form
\begin{equation*}
\Gamma(\rho)=\frac12\big(1\pm \tanh (\alpha(\rho^*-\rho))\big)
\end{equation*}
with constant $\alpha>0$ and $\rho^*>0$. It describes two complementary mechanisms for the phenotypic transmissions. 

 {\it Go-or-rest model.} Let us first look at a particular version of model 
\eqref{gog1}-\eqref{gog2} with no  
proliferation rate (namely $r=0$)
 which is called in \cite{Pham} as the  ``go-or-rest model'':
\begin{align}
u_t&=- \mu \Big(\Gamma(u+v)u-\big(1-\Gamma(u+v)\big)v\Big)
\label{gor1}
\\
v_t&= \Delta v +\mu \Big(\Gamma(u+v)u-\big(1-\Gamma(u+v)\big)v\Big).\label{gor2}
\end{align}
One can check by a simple calculation that this system supplemented with the Neumann boundary condition for $v(x,t)$ has a one parameter family of constant stationary solutions:
\begin{equation}\label{const:k}
\Big(k-\Gamma(k)k, \, \Gamma(k)k\Big) \qquad \text{for each fixed $k\in\R$}.
\end{equation}
These vectors are degenerate ({\it i.e.} they do not satisfy Assumption \ref{ass}) because the determinant 
in \eqref{non-deg} vanishes in this case.
However, by an elementary analysis of the phase portrait of the system of the ODEs,
\begin{equation}\label{ODE:k}
\frac{d}{dt}\u=- \mu \Big(\Gamma(\u+\v)\u-\big(1-\Gamma(\u+\v)\big)\v\Big),
\quad 
\frac{d}{dt}\v= \mu \Big(\Gamma(\u+\v)\u-\big(1-\Gamma(\u+\v)\big)\v\Big),
\end{equation}
one can show that vectors \eqref{const:k} are {\it stable solutions of system
\eqref{ODE:k}.}
The constant steady state \eqref{const:k} satisfies the autocatalysis condition \eqref{as:auto}
if 
\begin{equation}\label{auto:gor}
\Gamma'(k)k+\Gamma(k)<0,
\end{equation}
see \cite[Ch.~3.1]{Pham} for further discussion.
Thus, by our Theorem \ref{thm:non-const}, constant stationary solutions \eqref{const:k}  are 
{\it unstable solutions} of the 
reaction-diffusion-ODE system \eqref{gor1}-\eqref{gor2}.

System \eqref{gor1}-\eqref{gor2} has no heterogeneous stationary solutions, because the counterpart of 
the boundary-value problem \eqref{seq1}-\eqref{sNeumann} reduces in this case to the problem
$$
\Delta V=0 \quad \text{in $\Omega$,}\qquad \partial_\nu V=0\quad \text{on $\partial \Omega$},
$$
which has constant solutions, only.

{\it Go-or-grow model.} Let us now briefly sketch an analogous reasoning in the case of the more general model \eqref{gog1}-\eqref{gog2} with $r>0$. 
It has two constant stationary solutions ({\it cf.} \cite{Pham}):
$$
(\u,\v)=(0,0)\quad\text{and}\quad (\u,\v)=\big(1-\Gamma(1),\, \Gamma(1)\big).
$$
The nontrivial steady state $\big(1-\Gamma(1),\,\Gamma(1)\big)$ 
is a {\it stable solution} of the kinetic system corresponding to  \eqref{gog1}-\eqref{gog2} and
it satisfies the autocatalysis condition
\eqref{as:auto} if $(\Gamma'(1)+\Gamma(1))+(1-\Gamma(1))<0$ (see \cite{Pham}). 
In this case, 
by Theorem \ref{thm:non-const},
it is an {\it unstable solution} of the 
reaction-diffusion-ODE system \eqref{gog1}-\eqref{gog2}.

It is beyond the scope of this work to study positive heterogeneous stationary solutions
of the go-or-grow model with $r>0$. 
However, if there exist regular and strictly positive stationary solutions, 
then under the assumption  $(\Gamma'(1)+\Gamma(1))+(1-\Gamma(1))<0$,
they must be unstable by 
Theorems \ref{thm:non-const} and \ref{cor3}, see also Remark~\ref{rem:Turing}.
In conclusion, the structures shown in simulations of the models in ref.  \cite{Pham} are not Turing patterns.


\section{Instability of the stationary solutions}\label{sec:instab}

\subsection{Existence of solutions}
We begin our study of properties of solutions to the initial-boundary
value problem \eqref{eq1}--\eqref{ini} by recalling  results on
local-in-time existence and uniqueness of solutions for all bounded initial conditions. %
\begin{theorem}[Local-in-time solution]\label{thm:exist}
Assume that $u_0, v_0 \in L^\infty (\Omega)$. %
Then, there exists $T = T(\|u_0 \|_\infty,\, \|v_0 \|_\infty) > 0$ such
 that the initial-boundary value problem
 \eqref{eq1}--\eqref{ini} has a unique local-in-time mild
 solution $u, v \in L^\infty \big([0, T],\, L^\infty (\Omega)
 \big)$. %
\end{theorem}

We recall that a mild solution of problem \eqref{eq1}--\eqref{ini} is a pair of
measurable functions $u, v : [0, T] \times \overline{\Omega} \mapsto
\R$ satisfying the following system of integral equations
\begin{align}
u(x, t) &= u_0 (x) + \int_0^t f\big(u(x, s), v(x, s)\big)\, ds, \label{D1} \\
v(x, t) &= e^{t\Delta}v_0 (x) + \int_0^t e^{(t-s)\Delta} g\big(u(x, s), v(x,
 s)\big)\, ds,\label{D2}
\end{align}
where $e^{t\Delta}$ is the semigroup of linear operators generated by
Laplacian with the Neumann boundary condition. %
Since our nonlinearities $f = f(u, v)$ and $g = g(u, v)$ are locally
Lipschitz continuous, to construct a local-in-time unique solution of
system \eqref{D1}--\eqref{D2}, it suffices to apply the Banach fixed
point theorem. %
Details of such a reasoning and the proof of Theorem \ref{thm:exist}
in a case of much more general systems of reaction-diffusion equations
can be found {\it eg.} in \cite[Thm.~1, p.~111]{R84},
see also 
our recent work
\cite[Ch. 3]{MKS12} for a construction of nonnegative solutions of
particular reaction-diffusion-ODE problems. %

\begin{rem}
If $u_0$ and $v_0$ are more regular, {\it i.e.} if for some $\alpha\in (0,1)$ 
we have $u_0 \in C^\alpha
 (\overline{\Omega})$, $v_0 \in C^{2+\alpha}(\overline{\Omega})$
 and $\partial_\nu v_0 = 0$ on $\partial \Omega$, then the mild
 solution of problem  \eqref{eq1}--\eqref{ini} is smooth and satisfies $u \in C^{1, \alpha}\big([0, T]
 \times \overline{\Omega} \big)$ and $v \in C^{1 + \alpha/2,\, 2 +
 \alpha}\left([0, T] \times \overline{\Omega}\right)$.
We refer the reader  to \cite[Thm.~1, p.~112]{R84} as well as to \cite{GSV} for studies of general
reaction-diffusion-ODE systems in the H\"older spaces.
\end{rem}

\subsection{Linearization of reaction-diffusion-ODE problems.}
Let $(U,V)$ be a  stationary solution of problem 
\eqref{eq1}-\eqref{ini} 
--- either regular as discussed in Subsection \ref{sec2}  or weak (and possibly discontinuous) as defined in Subsection \ref{sec2.3}.
Substituting 
$$
u=U+\tu\qquad \text{and}\qquad v=V+\tv
$$
into \eqref{eq1}-\eqref{eq2}
we obtain the initial-boundary value problem for the perturbation
$(\tu,\tv)$ of the form \eqref{eq:a}:
\begin{equation}\label{perturb}
\frac{\partial}{\partial t}
\left(
\begin{array}{c}
\tu \\
\tv
\end{array}
\right)
=\L
\left(
\begin{array}{c}
\tu\\
\tv
\end{array}
\right)+
\N
\left(
\begin{array}{c}
\tu\\
\tv
\end{array}
\right),
\end{equation}
with the Neumann boundary condition,
$
\partial_\nu \tv=0,
$ 
where the linear operator $\L$ and the nonlinearity 
$\N$ are defined by formulas \eqref{lin} and \eqref{nonlin}, resp.

\begin{lemma}\label{lem:lin}
Let $(U,V)$ be a bounded (not necessarily regular) stationary solution of problem 
\eqref{eq1}-\eqref{ini}. We consider the following linear system
\begin{equation}\label{lin}
\left(
\begin{array}{c}
\tu_t \\
\tv_t
\end{array}
\right)
=\L
\left(
\begin{array}{c}
\tu\\
\tv
\end{array}
\right)
\equiv
\left(
\begin{array}{c}
0 \\
\Delta \tv
\end{array}
\right)
+
\left(
\begin{array}{cc}
f_u(U, V)&f_v(U, V)\\
g_u(U, V)&g_v(U, V)
\end{array}
\right)
\left(
\begin{array}{c}
\tu\\
\tv
\end{array}
\right)
\end{equation}
with the Neumann boundary condition
$
\partial_\nu \tv=0.
$ 
Then, for every $p\in (1,\infty)$,
the operator $\L$ with the domain $D(\L)=L^p(\Omega)\times W_N^{2,p}(\Omega)$ 
generates 
an analytic  semigroup $\{e^{t\L}\}_{t\geq 0}$ of linear operators on $L^p(\Omega)\times L^p(\Omega)$,
which satisfies ``the spectral mapping theorem'':
\begin{equation}\label{spec}
\sigma (e^{t\L})\setminus\{0\}=e^{t\sigma(\L)}\qquad 
\text{for every} \quad t\geq 0.
\end{equation}
\end{lemma}

\begin{proof}
Here, we use the Sobolev space 
$$
W_N^{2,p}(\Omega)= \{u\in W^{2,p}(\Omega)\;:\; \partial_\nu u=0 \quad \text{on}\quad 
\partial \Omega\}.
$$
Notice that $\L$ is a bounded perturbation of the operator
$$
\L_0
\left(
\begin{array}{c}
\tu\\
\tv
\end{array}
\right)
\equiv
\left(
\begin{array}{c}
0 \\
\Delta \tv
\end{array}
\right)
$$ 
with the domain $D(\L_0)=L^p(\Omega)\times W_N^{2,p}(\Omega)$, which generates an analytic 
semigroup on $L^p(\Omega)\times L^p(\Omega)$ for each $p\in (1,\infty)$.
Thus, it is well-known (see {\it e.g.} \cite[Ch.~III.1.3]{EN} and \cite[Theorems  2.15 and 2.19]{Y10}) that 
the same property holds true for the operator $(\L,D(\L))$.

The spectral mapping theorem for the semigroup
$\{e^{t\L}\}_{t\geq 0}$
expressed by equality \eqref{spec}
 holds 
true if the semigroup is {\it e.g.}  
eventually norm-continuous (see \cite[Ch.~IV.3.10]{EN}). 
Since every analytic semigroup of linear operators is eventually norm-continuous, we obtain immediately relation \eqref{spec}
(cf. \cite[Ch.~IV, Corollary 3.12]{EN}).
\end{proof}

Next, we show certain elementary estimate  of the nonlinearity in equation \eqref{perturb}.

\begin{lemma}\label{lem:nonlin}
Let $(U,V)$ be a bounded (not necessarily regular) stationary solution of problem 
\eqref{eq1}-\eqref{ini}.
Then, for every $p\in [1,\infty]$, the nonlinear operator
\begin{equation}\label{nonlin}
\N
\left(
\begin{array}{c}
\tu\\
\tv
\end{array}
\right)
\equiv
\left(
\begin{array}{c}
f(U+\tu,V+\tv)-f(U,V)\\
g(U+\tu,V+\tv)-g(U,V)
\end{array}
\right)
-
\left(
\begin{array}{cc}
f_u(U, V)&f_v(U, V)\\
g_u(U, V)&g_v(U, V)
\end{array}
\right)
\left(
\begin{array}{c}
\tu\\
\tv
\end{array}
\right)
\end{equation}
satisfies 
\begin{equation}\label{N1}
\left\|\N(\tu,\tv)\right\|_{L^p\times L^p}
\leq  C\big(\rho, \|U\|_{L^\infty}, \|V\|_{L^\infty}\big)\|(\tu,\tv)\|_{L^\infty\times L^\infty}
\|(\tu,\tv)\|_{L^p\times L^p}
\end{equation}
for all $\tu,\tv \in L^\infty$ such that $\|\tu\|_{L^\infty}<\rho$ and 
$\|\tv\|_{L^\infty}<\rho$, where $\rho>0$ is an arbitrary constant.
If, moreover, $U,W\in W^{1,p}(\Omega)$ then 
\begin{equation}\label{N2}
\begin{split}
\big\|\nabla &\N(\tu,\tv)\big\|_{L^p\times L^p}\\
&\leq  C\big(\rho, \|U\|_{L^\infty}, \|V\|_{L^\infty},  \|\nabla U\|_{L^p}, 
\|\nabla V\|_{L^p}\big) 
\|(\tu,\tv)\|_{L^\infty\times L^\infty}
\|(\nabla \tu,\nabla \tv)\|_{L^p\times L^p}
\end{split}
\end{equation}
for all $\tu,\tv \in L^\infty$ such that $\|\tu\|_{L^\infty}<\rho$ and 
$\|\tv\|_{L^\infty}<\rho$, where $\rho>0$ is an arbitrary constant.
\end{lemma}

\begin{proof}
The proofs of both inequalities consist in using 
the Taylor formula applied to the $C^3$-nonlinearities  
$f=f(u,v)$ and $g=g(u,v)$
  in problem \eqref{eq1}-\eqref{eq2}.
\end{proof}

\subsection{Continuous spectrum of the linear operator.}
Now, we are in a position to study the spectrum $\sigma(\L)$ of the linear operator 
$\L$, given by the formula \eqref{lin}
 when we linearize the reaction-diffusion-ODE problem \eqref{eq1}-\eqref{ini}
at a regular stationary solution.

\begin{theorem} \label{thm:spec:L}
Assume that  $(U(x),V(x))$ is a regular stationary solution of the  problem \eqref{eq1}-\eqref{ini} and
define the constants
\begin{equation}\label{l0}
 \lambda_0 = \inf_{x \in \overline{\Omega}} f_u \big(U(x), V(x)\big) 
\qquad 
\text{and}
\qquad 
\Lambda_0 = \sup_{x \in \overline{\Omega}} f_u \big(U(x), V(x)\big) > 0.
\end{equation}
Fix $p\in (1,\infty)$.
Let $\L$ be the  linear operator defined formally  by formula \eqref{lin}
with the domain $D(\L)=L^p(\Omega)\times W_N^{2,p}(\Omega)$.
Then 
$$[\lambda_0,\Lambda_0]\subset \sigma(\L).$$
\end{theorem}

\begin{proof}
 We show that  for each $\lambda\in  [\lambda_0,\Lambda_0]$ the operator 
$$\L-\lambda I: L^p(\Omega)\times W^{2,p}(\Omega)\to L^p(\Omega)\times L^p(\Omega)$$
defined by formula
$$
(\L-\lambda I)(\varphi,\psi) = \big( (f_u-\lambda)\varphi
 +f_v \psi,\  \Delta \psi +g_u\varphi+(g_v-\lambda)\psi\big),
$$
where $f_u=f_u\big(U(x),V(x)\big)$, {\it etc.},
cannot have a bounded inverse. Suppose, {\it a contrario}, that $(\L-\lambda I)^{-1}$ exists and is bounded. Then, for a constant $K= \| (\L-\lambda I)^{-1}\|$, we have
$$
\|( \varphi,\psi)\|_{L^p(\Omega)\times W^{2,p}(\Omega)}\leq
K \|(\L-\lambda I)( \varphi,\psi)\|_{L^p(\Omega)\times L^p(\Omega)}
$$
for all $( \varphi,\psi)\in {L^p(\Omega)\times W^{2,p}(\Omega)}$ or, equivalently,
using the usual norms in $L^p(\Omega)\times W^{2,p}(\Omega)$ and in $L^p(\Omega)\times L^{p}(\Omega)$:
\begin{equation}\label{i:instab}
\begin{split}
\|\varphi\|_{L^p(\Omega)} &+\|\psi\|_{W^{2,p}(\Omega)}\\
&\leq K\big(\|(f_u-\lambda)\varphi +f_v\psi\|_{L^p(\Omega)}+
\|\Delta \psi +g_u\varphi+(g_v-\lambda)\psi\|_{L^p(\Omega)}
\big).
\end{split}
\end{equation}
A contradiction will be obtained by showing that inequality \eqref{i:instab} cannot be true for all 
$( \varphi,\psi)\in {L^p(\Omega)\times W^{2,p}(\Omega)}$.

To prove this claim, first we observe that,  for each  $\lambda\in [\lambda_0,\Lambda_0]$, there exists $x_0\in \overline \Omega$ such that 
$f_u\big(U(x_0),V(x_0)\big)-\lambda=0$. Hence, for every $\varepsilon>0$ there is a ball $B_\varepsilon \subset \Omega$ such that $\|f_u-\lambda\|_{L^\infty(B_\varepsilon)}\leq \varepsilon.$

Next, for arbitrary $\psi \in C^\infty_c(\Omega)$ satisfying
 ${\rm supp\,}\psi\subset B_\varepsilon$, we  choose $\varphi\in L^p(\Omega)$ such that ${\rm supp\,}\varphi \subset B_\varepsilon$
and  in such a way that
$\Delta \psi +g_u\varphi+(g_v-\lambda)\psi= \zeta$, where the function $\zeta\in L^p(\Omega)$ satisfies 
$\|\zeta\|_{L^p(\Omega)}\leq \varepsilon \|\varphi \|_{L^p(\Omega)}$. 
Let us explain that such a choice of $\varphi,\zeta\in L^p(\Omega)$ is always possible. We cut $g_u$ at the 
level $\varepsilon$ in the following way
\begin{equation*}
g_u^\varepsilon = g_u^\varepsilon \big(U(x),V(x)\big) \equiv
\left\{
\begin{array}{ccc}
g_u \big(U(x),V(x)\big) &\text{if}& |g_u \big(U(x),V(x)\big)|>\varepsilon,\\
\varepsilon  &\text{if}& |g_u \big(U(x),V(x)\big)|\leq \varepsilon.
\end{array}
\right.
\end{equation*}
Thus, we obtain 
$$
\Delta \psi +g_u\varphi+(g_v-\lambda)\psi= \Delta \psi +g_u^\varepsilon\varphi+(g_v-\lambda)\psi
+ (g_u-g_u^\varepsilon)\varphi=\zeta$$ 
for
$$
\varphi 
= \frac{-\big(\Delta \psi +(g_v-\lambda)\psi\big)}{g_u^\varepsilon}\in L^p(\Omega)
\qquad\text{and}\qquad
\zeta=(g_u-g_u^\varepsilon)\varphi \in L^p(\Omega)
$$
with $\|g_u-g_u^\varepsilon\|_{L^\infty(\Omega)}\leq \varepsilon.$

Now, noting  that ${\rm supp\,}\varphi \subset B_\varepsilon$, we obtain the inequality
$$\|(f_u-\lambda) \varphi\|_{L^p(\Omega)}\leq 
\|(f_u-\lambda)\|_{L^\infty(B_\varepsilon)}\| \varphi\|_{L^p(\Omega)}\leq \varepsilon \| \varphi\|_{L^p(\Omega)}.
$$
Thus, substituting functions $\varphi$, $\psi$, and $\zeta$ into inequality \eqref{i:instab}, we obtain the estimate
\begin{equation}\label{ii:instab}
\begin{split}
\|\varphi\|_{L^p(\Omega)} &+\|\psi\|_{W^{2,p}(\Omega)}\\
 &\leq K\big(\|(f_u-\lambda)\|_{L^\infty(B_\varepsilon)} \|\varphi\|_{L^p(\Omega)} +
\|f_v \psi\|_{L^p(\Omega)}+\|\zeta\|_{L^p(\Omega)}\big)\\
&\leq K\big(2\varepsilon  \|\varphi\|_{L^p(\Omega)} +\|f_v\|_{L^\infty(\Omega)}\|\psi\|_{L^p(\Omega)}\big).
\end{split}
\end{equation}
Hence, choosing $\varepsilon>0$ 
in such a way that $2K\varepsilon\leq 1$
 and compensating the term $2K\varepsilon  \|\varphi\|_{L^p(\Omega)}$ on the right-hand side of  inequality \eqref{ii:instab}  by its counterpart on the left-hand side, we obtain the estimates
$$\|\psi\|_{W^{2,p}(\Omega)}\leq (1-2K\varepsilon)|\varphi\|_{L^p(\Omega)}+
\|\psi\|_{W^{2,p}(\Omega)}
\leq K\|f_v\|_{L^\infty(\Omega)}\|\psi\|_{L^p(\Omega)},
$$
which, obviously, cannot be true for all $\psi \in C^\infty_c(\Omega)$ such that ${\rm supp\,}\psi\subset B_\varepsilon$. 

We have completed the proof that each $\lambda\in [\lambda_0,\Lambda_0]$ belongs to $\sigma(\L)$.
\end{proof}

\subsection{Eigenvalues.}\label{sub:eigen}
In the Hilbert case  $D(\L)=L^2(\Omega)\times W_N^{2,2}(\Omega)$,
the remainder of the spectrum of $\big(\L, D(\L)\big)$ consists of 
a discrete  set of eigenvalues $\{\lambda_n\}_{n=1}^\infty \subset \C \setminus [\lambda_0,\Lambda_0]$. Here, we sketch the proof of this result, however, it does not play
any role in our instability results.

As the usual practice, 
we analyze  the corresponding resolvent equations 
\begin{align}
(f_u-\lambda)\varphi +f_v\psi&=F\qquad \text{in}\quad  \overline\Omega\label{res:lin:a}\\
\Delta \psi +g_u\varphi+(g_v-\lambda)\psi&=G \qquad \text{in}\quad  \Omega\label{res:lin:b}\\
\partial_\nu \psi&=0\qquad \text{on}\quad  \partial\Omega,\label{res:lin:c}
\end{align}
with arbitrary $F,G\in L^2(\Omega)$. 
Here, one should notice that  for every  $\lambda \in \C\setminus   [\lambda_0,\Lambda_0]$, 
one can solve equation \eqref{res:lin:a} with respect to $\varphi$.
Thus,  
after substituting the resulting expression $\varphi=(F-f_v\psi)/(f_u-\lambda)\in L^2(\Omega)$  
into \eqref{res:lin:b}, we obtain 
the boundary value problem
\begin{align}
&\Delta \psi + q ( \lambda) \psi = p(\lambda)&     \text{for}\quad &x\in {\Omega}, & \label{psi1r}\\
&\partial_\nu \psi = 0&  \text{for}\quad &x\in \partial\Omega,&\label{psi2r}
\end{align}
where 
\begin{equation}\label{pq}
 q(\lambda)=q(x, \lambda) = -\frac{g_u f_v}{f_u - \lambda} + g_v - \lambda
\qquad\text{and}\qquad
p(\lambda)=p(x, \lambda)= G- \frac{g_u F}{f_u - \lambda}.
\end{equation}
For a fixed $\lambda\in \C\setminus [\lambda_0,\Lambda_0]$,
by the Fredholm alternative, either the inhomogeneous  problem \eqref{psi1r}-\eqref{psi2r} has a unique solution 
(so, $\lambda$ is not an element of $\sigma(\L)$)
or 
else the homogeneous   boundary value problem
\begin{align}
&\Delta \psi + q ( \lambda) \psi = 0&     \text{for}\quad &x\in{\Omega}, & \label{psi1}\\
&\partial_\nu \psi = 0&  \text{for}\quad &x\in \partial\Omega,&\label{psi2}
\end{align}
has a nontrivial solution $\psi$. Hence, it suffices to consider those $\lambda \in \C\setminus  [\lambda_0,\Lambda_0]$,
for which problem \eqref{psi1}-\eqref{psi2} has nontrivial solution. 

Now, we are in a position to 
prove that the set $\sigma(\L)\setminus [\lambda_0,\Lambda_0]$ consists of isolated eigenvalues of $\L$, only.
Here, it suffices to  use 
the following general result on a family of compact operators, which we state
for the reader's convenience. The proof
can be  found in the Reed and Simon book 
\cite[Thm. VI.14]{reed-simon-I}.	

\begin{theorem}[Analytic Fredholm theorem]\label{Fredholm}
Assume that $H$ is a Hilbert space and denote by $L(H)$ the Banach space of all
bounded linear operators acting on  $H$. 
For an open connected set $D\subset \C$, let $f:D\to L(H)$ be an analytic operator-valued function such that $f(z)$ is compact for each $z\in D$. Then, either\\
(a) $(I-f(z))^{-1}$ exists for no $z\in D$, or\\
(b) $(I-f(z))^{-1}$ exists for all $z\in D\setminus S$, where $S$ is a discrete subset of $D$ {\rm (}{\it i.e.} a set which has no limit points in $D${\rm )}.
\end{theorem}

Let us rewrite problem \eqref{psi1}-\eqref{psi2} in the form
$$
\psi=  G\big[-(q(\lambda)+\ell)\psi\big] \equiv R(\lambda) \psi,
$$
where the operator $G=``{(\Delta-\ell I)^{-1}}$'' supplemented with the Neumann boundary conditions 
is defined in the usual way. Here, $\ell\in \R$ is a fixed number different from each eigenvalue of Laplacian with the Neumann boundary condition.

 Recall that, for each  $\lambda \in \C\setminus  [\lambda_0,\Lambda_0]$,
  the operator $R(\lambda):L^2(\Omega)\to L^2(\Omega)$ 
is compact as the superposition of the compact operator $G$ and of the continuous multiplication operator with the function $q(\lambda)+\ell \in L^\infty(\Omega)$. 
Moreover, the mapping $\lambda \mapsto R(\lambda)$ from the open set  $\C\setminus  [\lambda_0,\Lambda_0]$ into the Banach space of linear compact operators is analytic, which can be easily seen using the explicit form of $q(\lambda)$ in \eqref{pq}. 
Thus,  the set $\sigma(\L)\setminus [\lambda_0,\Lambda_0]$ consists of isolated points
due to  the analytic Fredholm Theorem \ref{Fredholm}.
Here, to exclude the case (a) in Theorem \ref{Fredholm}, we have to show that the operator 
$I-R(\lambda)$ is invertible for some $\lambda \in \C\setminus  [\lambda_0,\Lambda_0]$.
This is, however, the consequence of the fact that the inhomogeneous boundary value problem 
\eqref{psi1r}-\eqref{psi2r} has a unique solution if $\lambda >0$ is chosen so large that $q(x,\lambda)<0$.


\subsection{Linearization principle}
The next goal in this section is to recall that, under appropriate conditions, the linear 
instability of the stationary solutions of a reaction-diffusion-ODE problem implies their nonlinear
instability. 
Such a theorem is well-known for ordinary differential equations.
Furthermore,  in the case of reaction-diffusion equations where the spectrum of a linearized problem is discrete, one my  apply 
the abstract result from the book by Henry \cite[Thm. 5.1.3]{Henry}.
However, in the case of reaction-diffusion-ODE problems, 
the linearized operator at a stationary solution 
(either smooth or discontinuous) 
may have a non-empty  continuous spectrum ({\it cf.} Theorem \ref{thm:spec:L}).
Hence, checking the assumptions of general results
from \cite{Henry} does not seem to be straightforward.
Therefore, here, we propose a different approach.

Let us consider a general evolution equation
\begin{equation}\label{eq:a}
w_t=\L w +\N(w), \qquad w(0)=w_0
\end{equation}
where $\L$ is a generator of a $C_0$-semigroup of 
linear operators  $\{e^{t\L}\}_{t\geq 0}$  on a Banach space $X$
and $\N$ is  a nonlinear operator such that $\N(0)=0$.

First, we recall an idea  introduced by Shatah and 
Strauss \cite{SS00} which asserts that, under relatively strong assumption on a nonlinearity 
in  equation \eqref{eq:a}, 
the existence of a positive part of  the spectrum of the linear operator $\L$ is sufficient to show  that the zero solution of equation \eqref{perturb}
 is unstable. This is  the precise statement of that result.

\begin{theorem}[{\cite[Thm 1]{SS00}}]\label{thm:SS}
Consider an abstract problem \eqref{eq:a}, where
\begin{enumerate}
\item the linear operator $\L$ generates a strongly continuous semigroup of linear operators on a Banach space $X$,
\item the intersection of the spectrum of $\L$ with the right half-plane $\{\lambda\in \C:\;: \Re \lambda>0\}$ is nonempty.   
\item $\N :X\to X$ is continuous and there exist constants $\rho>0$, $\eta>0$, and $C>0$ 
such that $\|\N (w)\|_X \leq C\|w\|_X^{1+\eta}$ for all $\|w\|_X<\rho$.
\end{enumerate} 
Then the zero solution of this equation is (nonlinearly) unstable.
\end{theorem}

We apply Theorem \ref{thm:SS} to show an instability of 
regular steady states.
In the case of discontinuous stationary solutions, 
we are unable  to show that the nonlinearity 
in equation \eqref{perturb}
satisfies
 the the condition (3) of Theorem \ref{thm:SS}.
One may overcome this obstacle by assuming that the 
  the spectrum  $\sigma(\L)$ has so-called spectral gap. 
This classical method has been recently used by Mulone and Solonnikov \cite{MS09} 
to show the instability of regular stationary solutions to certain reaction-diffusion-ODE problems, however, assumptions imposed in \cite{MS09} are not satisfied in our case.

The crucial idea underlying this approach  is to use two Banach
spaces: a ``large'' space $Z$ where 
the spectrum of a linearized operator is studied
 and a ``small'' space  $X\subset Z$ where an existence of solutions 
 can be proved. 
More precisely,
let $(X,Z)$ be a pair of Banach spaces such that
$X\subset Z$ with a dense and continuous embedding.
A solution $w\equiv 0$ of the Cauchy problem
\eqref{eq:a}  is called {\it $(X,Z)$-nonlinearly stable} if for every $\varepsilon > 0$, 
there exists $\delta > 0$
so that if $w(0) \in X$  and $\|w(0)\|_Z < \delta$,   then
\begin{enumerate}
\item there exists a global in time solution of  \eqref{eq:a}
such that $w \in C([0,\infty);X)$;
\item $\|w(t)\|_Z < \varepsilon	$ for all $t \in [0,\infty)$.
\end{enumerate}
An equilibrium $w\equiv 0$ that is not stable (in the above sense) is called Lyapunov
unstable.

In this work, we drop the reference to the pair $(X,Z)$.
Let us also note that, under this definition of stability, a loss of the existence
of a solution  of  \eqref{eq:a}  is a particular case of instability.

Now, we recall a result linking the existence of the so-called {\it spectral gap} to the nonlinear instability of a trivial solution
to problem \eqref{eq:a}.

\begin{theorem}\label{thm:spec}
We impose the two following assumptions.
\begin{enumerate}
\item The semigroup of linear operators $\{e^{t\L}\}_{t\geq 0}$ on $Z$ 
satisfies ``the spectral gap condition'', namely, we suppose that for every $t>0$ the spectrum $\sigma$ 
of the linear operator $e^{t\L}$ can be decomposed as 
follows: $\sigma =\sigma(e^{t\L})= \sigma_{-}\cup\sigma_{+} $ with $\sigma_+\neq \emptyset$, where
$$
\sigma_{-}\subset \{z\in \C\;:\; e^{\kappa t}<|z|<e^{\mu t}\}
\quad\text{and}\quad
\sigma_{+}\subset \{z\in \C\;:\; e^{Mt}<|z|<e^{\Lambda t}\}
$$
and 
$$
-\infty\leq \kappa<\mu<M<\Lambda<\infty\qquad \text{for some} \quad M>0.
$$
\item The nonlinear term $\N$ satisfies the inequality
\begin{equation}\label{as:N}
\|\N(w)\|_Z\leq C_0 \|w\|_X\|w\|_Z \qquad \text{for all}\quad w\in X \quad \text{satisfying}
\quad \|w\|_X<\rho
\end{equation}
for some constants $C_0>0$ and $\rho>0$.
\end{enumerate}
Then, the trivial solution $w_0\equiv 0$ of the Cauchy problem \eqref{eq:a} is nonlinearly unstable.
\end{theorem}

The proof of this theorem can be found  in the work by Friedlander {\it et al.} 
\cite[Thm. 2.1]{FSV}.

\begin{rem}\label{rem:spec}
The operator $\L$ considered in this work satisfies the ``spectral mapping theorem":
$\sigma (e^{t\L})\setminus\{0\}=e^{t\sigma(\L)}$, see Lemma \ref{lem:lin}.
Thus, due to the relation $|e^z|=e^{\Re z}$ for every $z\in\C$, the spectral gap condition required in Theorem \ref{thm:spec} holds true if for every $\lambda\in \sigma(\L)$, either $\Re\lambda \in (\kappa,\mu)$ or $\Re\lambda \in (M,\Lambda)$.
\end{rem}

\begin{rem}
The authors of the reference \cite[Thm. 2.1]{FSV} formulated their instability result
under the spectral gap condition for a {\it group} of linear operators  
$\{e^{t\L}\}_{t\in\R}$ and in the case 
of a finite constant $\kappa$ (caution: in \cite{FSV},  the letter $\lambda$ is used 
instead of $\kappa$).
However, the proof of \cite[Thm. 2.1]{FSV} holds true (with a minor and obvious modification) in the case of a semigroup $\{e^{t\L}\}_{t\geq 0}$ as well as $\kappa=-\infty$ is allowed, as stated in Theorem \ref{thm:spec}.
This extension is important to deal with the operator $\L$ introduced in 
Lemma \ref{lem:lin}, which generates a semigroup of linear operators, only, 
and which may have an unbounded sequence of eigenvalues.
\end{rem}

\subsection{Proofs of instability results}\label{subsec:instab}

\begin{proof}[Proof of Theorem \ref{thm:non-const}]
Here, it suffices to apply  Theorem \ref{thm:SS}
to the semi-linear equation \eqref{perturb} with the Banach space 
$$
X=W^{1,p}(\Omega)\times W^{1,p}(\Omega) \qquad \text{for some}\quad 
p>n.
$$
Recall the well-known embedding $W^{1,p}(\Omega)\subset L^\infty(\Omega)$ for every $p\in (n,\infty]$..

We refer the reader to \cite[Ch.~2]{Y10} for the proof that the operator $\L$ discussed in Lemma~\ref{lem:lin} generates a semigroup of linear operators on $X$.
The autocatalysis condition \eqref{as:auto} combined with Theorem \ref{thm:spec:L}
imply that $\sigma(\L)$ meets the right-hand plane of $\C$.
Due to the embedding $X\subset L^\infty(\Omega)\times L^\infty(\Omega)$,
inequalities \eqref{N1} and \eqref{N2} imply that the nonlinear mapping 
$\N$ in \eqref{nonlin} satisfies the condition (3) of Theorem \ref{thm:SS} with $\eta=1$.

Hence, the regular stationary solution $(U,V)$ is unstable.
\end{proof}

\begin{proof}[Proof of Theorem  \ref{thm:weak}]
To show an instability of non-regular stationary solution, we begin 
as in the proof of Theorem \ref{thm:non-const}. First, we linearize our 
problem at a weak bounded stationary solution $(U,V)$ and 
we notice that assumptions of
Lemmas \ref{lem:lin} and \ref{lem:nonlin} are satisfied. 
Next, following the arguments from the proof of Theorem \ref{thm:spec:L}
we show that the number $f_u(U(x_0), V(x_0))$
belongs to $\sigma(\L)$, where $f_u(U(x),V(x))$ is positive at $x_0$ and continuous in its neighborhood. Notice that we do not need to show that all numbers from ${\rm Range}\, f_u(U,V)$ are in $\sigma(\L)$ to show the spectral gap condition required by Theorem 
\ref{thm:spec}. 
The reasoning from Subsection \ref{sub:eigen} concerning eigenvalues 
can be copied here without any change because $q(\lambda,x)$ defined in \eqref{pq}
is a bounded function for every $\lambda\in \C\setminus [\lambda_0,\Lambda_0]$.

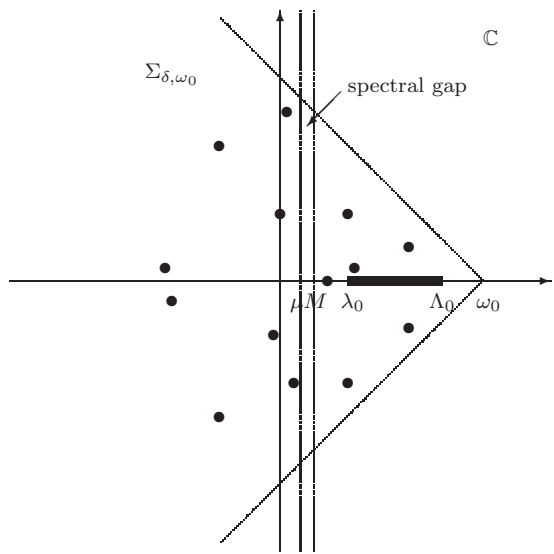
\begin{figure}  \label{fig1}
  \setlength{\unitlength}{0.9mm}

{\tiny

\begin{picture}(100,100)

\put(50,50){\vector(0, 1){40}}
\put(50,50){\vector(1, 0){40}}
\put(50,50){\line(-1, 0){40}}
\put(50,50){\line(0, -1){40}}

\put(80,85){$\mathbb{C}$}
\put(30,80){${\Sigma_{\delta,\omega_0}}$}

\multiput(80,50)(-0.3,0.3){130}{\line(0,1){0.2}}
 \multiput(80,50)(-0.3,-0.3){130}{\line(0,1){0.2}}

\multiput(53,10)(0,0.3){267}{\line(0,1){0.1}}

\multiput(55,10)(0,0.3){267}{\line(0,1){0.1}}

\put(59,46){$\lambda_0$}
\put(72,46){$\Lambda_0$}
\put(79,46){$\omega_0$}

\put(51.5,46){$\mu$}
\put(53.5,46){$M$}

\put(60,78){spectral gap}
\put(59,78){\vector(-1,-1){5}}

\linethickness{1.2mm}
\put(60,50){\line(1, 0){14}}
\put(69,55){\circle*{1.5}}
\put(57,50){\circle*{1.5}}
\put(61,52){\circle*{1.5}}
\put(50,60){\circle*{1.5}}
\put(41,70){\circle*{1.5}}
\put(33,52){\circle*{1.5}}
\put(51,75){\circle*{1.5}}
\put(60,60){\circle*{1.5}}

\put(69,43){\circle*{1.5}}
\put(49,42){\circle*{1.5}}
\put(41,30){\circle*{1.5}}
\put(34,47){\circle*{1.5}}
\put(52,35){\circle*{1.5}}
\put(60,35){\circle*{1.5}}
  \end{picture}
  }
  \caption{The spectrum $\sigma(\mathcal{L})$ is marked by thick dots and by the interval $[\lambda_0,\Lambda_0]$ in the sector $\Sigma_{\delta,\omega_0}$.
The spectral gap is represented by the strip
$\{\lambda\in\mathbb{C}\,:\, \mu\leq {\rm Re}\, \lambda \leq M\}$ without elements of $\sigma(\mathcal{L})$.
}
\end{figure}

Now, let us show that the operator $\L$ has a spectral gap as required in assumption (1) of
Theorem~\ref{thm:spec}.

By Lemma \ref{lem:lin}, there exists a number $\omega_0\geq 0$ such that the operator 
$\big(\L-\omega_0 I, D(\L)\big)$ generates a bounded analytic semigroup on 
$L^2(\Omega)\times L^2(\Omega)$; hence, this is a sectorial operator, see 
\cite[Ch.~II, Thm.~4.6]{EN}. In particular, there exists $\delta \in (0,\pi/2]$ such that 
$
\sigma(\L)\subset \Sigma_{\delta,\omega_0}\equiv 
\{\lambda\in\C\;:\; |{\rm arg}\; (\lambda-\omega_0)|\geq \pi/2+\delta\},
$
see Fig. \ref{fig1}.
A part of the spectrum 
$\sigma(\L)$ in the triangle $\Sigma_{\delta,\omega_0}\cap \{\lambda\in \C\;:\: {\rm Re}\; \lambda>0\}$
consists  of   the interval $[\lambda_0,\Lambda_0]$ where $\lambda_0>0$
and of a discrete  set
of  eigenvalues  
(by discussion in Subsection~\ref{sub:eigen})
with accumulation points from the interval 
$[\lambda_0,\Lambda_0]$, only (by Theorem \ref{Fredholm}.b).
Thus, we can easily find infinitely many $0\leq \mu<M\leq \lambda_0,$
for which the spectrum $\sigma(\L)$ can be decomposed as 
 required in Theorem~\ref{thm:spec}.  Here, one should use  the spectral mapping theorem, {\it i.e.} equality  \eqref{spec}, and 
 Remark \ref{rem:spec}.

Now, to complete the proof of an instability of not-necessarily regular stationary solutions,
we apply Theorem \ref{thm:spec} with
$X=L^\infty(\Omega)\times L^\infty(\Omega)$ and  $Z=L^2(\Omega)\times L^2(\Omega)$
for a bounded domain $\Omega\subset \R^N$ with a regular boundary, supplemented with the usual norms.
Then, required estimate of the nonlinear mapping in \eqref{as:N}
is stated in inequality \eqref{N1} with $p=2$.
\end{proof}


\begin{proof}[Proof of Corollary \ref{cor:weak}]
Here, the analysis is similar  to the case of regular stationary solutions discussed 
in Theorem \ref{thm:non-const}, hence, we only emphasize the most important steps.

Let $(U,V)$ be a weak solution of problem \eqref{seq1}-\eqref{sNeumann} and denote by 
 $\I\subset \overline\Omega$ its {\it null} set, namely, a measurable set such that $U(x)=0$ for all $x\in \I$ 
and $U(x)\neq 0$ for all $x\in \overline\Omega \setminus \I$.
For a null set $\I$, we define the associate $L^2$-space 
$$
L^2_\I(\Omega)=\{ v\in L^2(\Omega )\,:\, v(x)=0\quad \text{for}\quad x\in\I\}, 
$$
supplemented with the usual $L^2$-scalar product, which  is a Hilbert space as the closed subspace of $L^2(\Omega)$.
In the same way, we define the  subspace $L^\infty_\I(\Omega)\subset L^\infty(\Omega)$ by the equality 
$L^\infty_\I(\Omega)=\{ v\in L^\infty(\Omega )\,:\, v(x)=0\; \text{for}\; x\in\I\}$.

Obviously,  when the measure of $\I$ equals zero, we have $L_\I^2(\Omega)=L^2(\Omega)$.
The imposed assumptions  imply that 
$\I$  is different from the whole interval.

Now, observe that if $u_0(x)=0$ for some $x\in \Omega$ then by equations \eqref{eq1}
with the nonlinearity $f(u,v)=r(u,v)u$
 we have $u(x,t)=0$ for all $t\geq 0$. 
 Hence,   the spaces 
\begin{equation}\label{HI}
X_\I=L^\infty_\I(\Omega)\times  L^\infty(\Omega) \qquad \text{and}\qquad Z_\I=L^2_\I(\Omega)\times  L^2(\Omega)
\end{equation}
are  invariant 
 for the flow generated by problem \eqref{eq1}-\eqref{ini} (notice that there is no ``$\I$'' in the second coordinates of $X_\I$ and $Z_\I$). 
 The crucial part of our analysis is based on the fact that, as long as we work in the space $X_\I$ and $Z_\I$,
 we can linearize problem \eqref{eq1}-\eqref{ini} at the weak solution  $(U, V)$. 
Moreover, for each $x\in \overline\Omega \setminus\I$, 
the corresponding linearized operator agrees with $\L$ defined in Lemma \ref{lem:lin}.
 Hence, the analysis from the proof of Theorem~\ref{thm:non-const} can be directly adapted to discontinuous steady states in the following way.

We fix  a weak stationary solution $(U_\I, V_\I)$    
with a null set $\I\subset \Omega$.
The Fr\'echet  derivative of the nonlinear mapping $\F: Z_\I\to Z_\I$  defined by the mappings $(U,V)\mapsto (f(U,V), g(U,V))$ 
at the point $(U_\I, V_\I)\in Z_\I$
has the form
$$
D\F(U_\I, V_\I)
\left(
\begin{array}{c}
 \varphi    \\
\psi        
\end{array}
\right)
=\A_\I(x)
\left(
\begin{array}{c}
 \varphi    \\
\psi        
\end{array}
\right),
$$
where
$$   
\A_\I(x) =
\left(
\begin{array}{cc}
f_u(U_\I(x), V_\I(x))&f_v(U_\I(x), V_\I(x))\\
g_u(U_\I(x), V_\I(x))&g_v(U_\I(x), V_\I(x))
\end{array}
\right)
$$
This results immediately from the definition of the Fr\'echet  derivative.

 Next, we study spectral properties  of 
the linear operator 
$$
\L_\I
\left(
\begin{array}{c}
 \varphi    \\
\psi       
\end{array}
\right)
=
\left(
\begin{array}{cc}
       0       &  0\\
    0     &   \Delta \psi
\end{array}
\right)+\A_\I(x)
\left(
\begin{array}{c}
 \varphi    \\
\psi         
\end{array}
\right),
$$
 in the Hilbert space $Z_\I$  (see \eqref{HI}) with the domain 
$ 
D(\L_\I)= L_\I^2(\Omega)\times W^{2,2}(\Omega).
$
Here, the reasoning from the proof of Theorem \ref{thm:non-const} can be directly adapted with  modifications as in the proof of Theorem  \ref{thm:weak}.

Finally, we may study the discrete spectrum of $\L_\I$ in the same way as in 
Subsection~\ref{sub:eigen} because the corresponding function $q(\lambda,x)$
is bounded for $\lambda \in \C\setminus [\lambda_0,\Lambda_0]$.
The proof of  instability of the stationary solution $(U_\I, V_\I)$ is completed by Theorem
\ref{thm:spec} and Lemmas~\ref{lem:lin}-\ref{lem:nonlin}. 
\end{proof}

\section{Constant steady states which are intersected by  non-constant stationary solutions}\label{sec:const:non}
First, we prove  a certain property of  stationary solutions to a general
elliptic Neumann problem. This result will imply immediately Proposition \ref{thm:ineq}.

\begin{theorem}\label{lem:N}
Assume that 
$V\in C^2(\Omega)\cap C^1(\bOmega)$ is a non-constant solution of the 
boundary value problem
\begin{eqnarray}
\Delta V+h(V)=0 \quad \text{for}\ x \in \Omega\quad \text{and} \quad \partial_\nu V=0
 \quad \text{for}\ x \in \partial\Omega. \label{lem4.3-eq1}
\end{eqnarray}
Then, there exists $x_0\in \bOmega$ and $a_0\in \R$ such that
\begin{equation}
V(x_0)=a_0, \quad h(a_0)=0 \quad\text{and}\quad  h'(a_0)\geq 0. \label{lem4.3-eq10}
\end{equation}
\end{theorem}

\begin{proof}
First, as in the proof of Proposition \ref{prop1}, 
we integrate the equation in \eqref{lem4.3-eq1} and we use the Neumann
 boundary condition to obtain
$
 \int_\Omega h(V(x))\, dx = 0.
$
Hence,  there exists $x_0 \in \overline{\Omega}$ and $a_0\in \R$ such that
$
V(x_0)=a_0$
 and
$
h(a_0)=0.
$
Now, we suppose that 
 $h^\prime (a_0) < 0$, and consider two cases: $x_0 \in \Omega$ and $x_0
 \in \partial \Omega$, separately. %

Let $x_0 \in \Omega$. Since $h(a_0) = 0$, we have
\[
 \Delta (V - a_0) + h(V) - h(a_0) = 0.
\]
Using the well-known formula
\begin{align*}
 h(V) - h(a_0) &= \int_0^1 \frac{d}{ds}h\big(s V + (1-s)a_0\big)\, ds \\
&= (V-a_0)\int_0^1 h^\prime \big(sV + (1-s)a_0\big)\, ds,
\end{align*}
we obtain
\begin{align}
\Delta (V - a_0) + r(x, a_0) (V - a_0) = 0, \label{lem4.3-eq2}
\end{align}
where $r(x, a_0) = \int_0^1 h^\prime \big(sV(x) + (1-s)a_0\big)\, ds$. %
Observe that $r(\cdot, a_0) \in C(\Omega)$ and 
\begin{align*}
r(x_0, a_0) &= \int_0^1 h^\prime (sV(x_0) + (1-s)a_0)\, ds \\
&= \int_0^1 h^\prime (s a_0 + (1-s)a_0)\, ds \\
&= h^\prime (a_0) < 0. 
\end{align*}
Hence, there exists an open neighbourhood $\mathcal{U} \subset \Omega$ of $x_0$
 such that %
$r (x, a_0) < 0$ for all $x \in \mathcal{U}$. %
Suppose that $r (x, a_0) < 0$ for all $x \in \Omega$. %
Multiplying both sides of equation \eqref{lem4.3-eq2} by $V(x)-a_0$ and
 integrating over $\Omega$, we obtain
\[
- \int_\Omega |\nabla (V(x)-a_0)|^2\, dx + \int_\Omega r(x, a_0) (V(x) -
 a_0)^2\, dx = 0.
\]
This implies that $V(x) \equiv a_0$, which is a contradiction, because we assume that $V=V(x)$ is a non-constant solution. %
Therefore, there exists $x_1 \in \partial \mathcal{U} \cap \Omega$ such that $r(x_1, a_0) = 0$. %
It follows from equation \eqref{lem4.3-eq2} that $\Delta V(x_1) = 0$,
 and consequently, from equation \eqref{lem4.3-eq1} we have $h(V(x_1)) = 0$. %
Hence, there exists $a_1 \in \R$ such that
 $V(x_1) = a_1$ and $h(a_1) = 0$. %
Note that $a_1 \neq a_0$. %
Thus, if the equation $h(V)=0$ has only one root, the proof of Theorem
 \ref{lem:N} is completed. 

Now, we consider two cases.

\noindent
{\it Case I:\ The equation $h(V) = 0$ has no solution between $a_0$ and
 $a_1$.} %
Thus, we define the function
\[
 \psi(s) = V(x_0 + s (x_1 - x_0))\qquad  \text{for}\  0 \le s
 \le 1,
\] %
and, without loss of generality, we can assume that $a_0 < \psi(s) < a_1$ for all $0 <
 s < 1$. %

Since $h(a_0) = 0$ and $h^\prime (a_0) < 0$, we have $h(a_0 + \theta_0) <
 0$ for small $\theta_0 > 0$. If we also suppose that $h^\prime (a_1) <
 0$, then, we can find small $\theta_1 > 0$ such that $h(a_1 -
 \theta_1) > 0$. %
Noting that $\psi (s)$ is continuous function, we see that there exist
 $s_\ast,\ s_{\ast \ast} \in (0, 1)$ such that $\psi (s_\ast) = a_0 + \theta_0$ and
 $\psi (s_{\ast \ast}) = a_1 - \theta_1$. %
This implies that there exists $ \hat{s} \in (s_\ast, s_{\ast \ast})$
 for which $h(V(x_0 + \hat{s} (x_1 - x_0))) = 0$, %
and from the assumption for $\psi(s)$, 
\[
a_0 < V(x_0 + \hat{s} (x_1 - x_0)) < a_1.
\]
This is a contradiction with the hypothesis that the equation $h(V) = 0$
 has no roots between $a_0$ and $a_1$. %
Hence, $h^\prime (a_1) \ge 0$. %

\noindent
{\it Case II:\ The equation $h(V) = 0$ has a solution $a_m$ between $a_0$ and
 $a_1$.} %
It is clear that $V(x)$ has to intersect $a_m$, too. %
Choosing $a_m$ the closest root of $h(V) = 0$ to $a_0$, we repeat the argument from Case I
 to show that $h^\prime (a_m) \ge 0$. %

Next, let $x_0 \in \partial \Omega$. %
Following the previous reasoning and using the hypothesis 
$h'(a_0)<0$,
we find a ball $B \subseteq \Omega$ such that $x_0 \in \partial
 \Omega$ and $r(x, a_0) < 0$ for all $x \in B$. %
Moreover, we can assume that either
$V(x) > a_0$ or $V(x) < a_0$ for all $x \in
 B$, because, if there exists $x_1 \in B$ such that 
$V(x_1) = a_0$, then we
 can apply the same argument as in the first part of this proof
 to obtain 
$h^\prime (a_0)\geq 0$. %

Let $V(x) < a_0$ for all $x \in B$, and we apply the Hopf boundary lemma to equation \eqref{lem4.3-eq2} in
 the ball $B$. %
If $V$ is a non-constant solution satisfying $V(x) - a_0 < 0$ and
 $V(x_0) - a_0 = 0$, then necessarily
$
 {\partial V(x_0)}/{\partial \nu} > 0,
$
which contradicts the Neumann boundary condition satisfied by $V$ at
 $x_0 \in \partial \Omega$. %

Now, we consider the case $V(x) > a_0$ for all $x \in B$. %
Here, the function $U (x) = -(V(x) - a_0)$ satisfies the equation
\begin{align*}
-\Delta U + (- r(x, a_0))U = 0 \qquad \text{in}\ B 
\end{align*}
where  $r(x, a_0) < 0$, $U(x) < 0$ for all $x\in B$
and $U(x_0)= 0$. Hence,  the
 Hopf boundary lemma  yields a contradiction. %

Thus, $h'(a_0)\geq 0$ and the proof is complete.
\end{proof}


\begin{proof}[Proof of Proposition \ref{thm:ineq}.]
Let $\big(U(x), V(x)\big)$ be
a non-constant regular stationary solution  of \eqref{eq1}--\eqref{Neumann}
which means that  $U = k(V)$ and $f\big(k(V), V \big) =0$.
Since each constant solution  is non-degenerate, the equation 
$f(U, V) =0$ can be solved locally with respect to $U$, which implies that $k=k(V)$ is a $C^1$-function. 
Substituting $U = k(V)$ into equation \eqref{seq2} and %
denoting $h(V) = g\big(k(V), V \big)$, we obtain the following boundary value problem
satisfied by $V=V(x)$:
\begin{align}
\Delta V + h(V) = 0 \quad &\text{for}\ x \in \Omega,\label{th3.3-eq1}\\
\partial_\nu V = 0 \quad &\text{for}\ x \in \partial \Omega.
\label{th3.3-ini}
\end{align}
By Theorem  \ref{lem:N}, there exists a constant solution $\v$ of problem 
\eqref{th3.3-eq1}-\eqref{th3.3-ini}
and  $x_0\in \bOmega$ such that
\begin{equation}
V(x_0)=\v, \quad h(\v)=0, \quad\text{and}\quad  h^\prime (\v) \geq 0. \label{th3.3-eq2}
\end{equation}

On the other hand, differentiating the function $h(v) = g\big(k(v), v\big)$  yields
\begin{equation}
 h^\prime (v) = k^\prime (v)g_u \big(k(v), v \big) + g_v \big(k(v), v \big). \label{th3.3-eq3}
\end{equation}
Moreover, we differentiate both sides of the equation $f \big(k(v),
 v \big) = 0$ with respect to $v$ to obtain 
$ k^\prime (v) f_u \big(k(v), v \big) + f_v \big(k(v), v \big) = 0$. %
Hence, 
\begin{equation}
k^\prime (v) = - \frac{f_v \big(k(v), v \big)}{f_u \big(k(v), v \big)}. \label{th3.3-eq4}
\end{equation}
Finally, choosing $v=\v$ and $u=\u=k(\v)$ and 
substituting equation \eqref{th3.3-eq4} into \eqref{th3.3-eq3}, 
we obtain
\begin{align*}
 h^\prime (\v) &= - \frac{f_v \big(k(\v), \v \big)}{f_u \big(k(\v), \v \big)}g_u
 \big(k(\v), \v \big) + g_v \big(k(\v), \v \big)\\
& = \frac{1}{f_u (\u, \v)}\big[f_u (\u, \v) g_v (\u, \v) - f_v (\u, \v)
 g_u (\u, \v) \big]\\
&=\frac{1}{f_u(\u,\v)}\det 
\left(
\begin{array}{cc}
f_u(\u,\v)&f_v(\u,\v)\\
g_u(\u,\v)&g_v(\u,\v)\\
\end{array}
\right).
\end{align*}
By \eqref{th3.3-eq2}, it holds $h'(\v)\geq0$. Moreover,  
since $(\u, \v)$ is non-degenerate, we obtain $h'(\v)>0$.
This completes the proof of inequality \eqref{ineq}.
\end{proof}




\appendix

\section{Model of early carcinogenesis -- quasi-stationary approximation }

In Appedix, we discuss in detail the model example from Subsection \ref{sub:carc}.

Initial-boundary value problems for  reaction-diffusion-ODE systems  arise in the modeling of the growth of a spatially-distributed cell population, where proliferation is controlled by endogenous or exogenous growth factors diffusing in the extracellular medium and binding to cell surface as proposed by Marciniak and Kimmel in the series of recent papers \cite{MCK06, MCK07, MCK08}.
Here, we consider the following particular case of such  systems,
which was studied by us in \cite{MKS12}
\begin{alignat}{2}
\partial_t u_{\ve}&=\Big(\frac{a v_{\ve}}{u_{\ve}+v_{\ve}} -d_c\Big) u_{\ve} &&\equiv f(u_{\ve},v_{\ve}), \label{eq1epA}\\
\ve \partial_t v_{\ve} &=-d_b v_{\ve} +u_{\ve}^2 w_{\ve} -d v_{\ve}&&\equiv g(u_{\ve},v_{\ve},w_{\ve}), \label{eq2epA}\\
\partial_t w_{\ve}-D\Delta w_{\ve} +d_g w_{\ve}&=  -u_{\ve}^2 w_{\ve} +d v_{\ve} +\kappa_0&&\equiv h(u_{\ve},v_{\ve},w_{\ve}), \label{eq3epA}
\end{alignat}
on $(0,\infty)\times \Omega$, supplemented with zero-flux boundary conditions for $w_{\ve}$
\begin{equation}\label{NepA}
\partial_\nu w_{\ve}(x,t)=0 \quad \mbox{for all} \quad t>0, \:\:x\in \partial \Omega
\end{equation}
and with nonnegative, smooth and bounded initial data
\begin{equation}
u_{\ve}(x,0)=u_0(x), \quad v_{\ve}(x,0)=v_0(x), \quad w_{\ve}(x,0)=w_0(x).\label{iniepA}
\end{equation}
In equations \eqref{eq1epA}-\eqref{eq3epA}, the letters 
$a , d_c,d_b, d_g, d,  \kappa_0,D $  denote positive constants.

As it was shown in \cite[Sec.~3]{MKS12}, solutions of this problem are nonnegative and stay bounded on every interval $[0,T]$. Here, we study the behavior of these solutions as $\ve\rightarrow 0$.

First, we notice that choosing $\ve= 0$ in equation \eqref{eq2epA}, we obtain 
the identity
\begin{equation}\label{v_quasi}
v=\frac{u^2 w}{d_b+d}.
\end{equation}
Substituting it to the two remaining equations 
\eqref{eq1epA}, \eqref{eq3epA}
yields the system
\begin{alignat}{2}
&u_t=\Big(\frac{a uw}{d_b+d+uw} -d_c\Big) u &&\text{for}\ x\in \overline\Omega, \; t>0, \label{eqn1rA}\\
&w_t = D\Delta w -d_g w -\frac{d_b}{d_b+d} u^2 w  +\kappa_0 && \text{for}\ x\in \Omega, \; t>0, \label{eqn2rA}\\
&\partial_\nu w (x,t)=0 && \text{for} \ x\in \partial \Omega, \;\; t>0,\label{NrA}\\
&u(x,0)=u_0(x)\quad \text{and}  \quad w(x,0)=w_0(x).&& \label{inirA}
\end{alignat}
Repeating our reasoning from \cite{MKS12} one can show that 
this new  system has also a unique, global-in-time  
nonnegative, smooth solution for  nonnegative and smooth initial data. 

In this part of Appendix, we show that  solutions of the quasi-stationary system \eqref{v_quasi}-\eqref{inirA} are good approximation of  solutions of the original three-equation model  \eqref{eq1epA}-\eqref{iniepA}. Our main result concerns uniform estimates for an approximation error for $u$ and $w$ 
on each finite time interval $[0,T]$.

First, we show that solutions of $\ve$-problem \eqref{eq1epA}-\eqref{iniepA} are uniformly bounded with respect to small $\ve>0$, locally in time.

\begin{lemma}\label{lem:A}
For every $T>0$ there exists $C(T)>0$  such for all sufficiently small $\varepsilon>0$  
{\rm (}e.g. $\varepsilon \in \big(0, (d_b+d)/(2d_g)\big)${\rm )}
the solution of system \eqref{eq1epA}-\eqref{iniepA} satisfies
$$
\|u_{\ve}(t)\|_{L^{\infty}(\Omega)}\leq C(T), \qquad \|v_{\ve}(t)\|_{L^{\infty}(\Omega)}\leq C(T), \qquad \|w_{\ve}(t)\|_{L^{\infty}(\Omega)}\leq C(T)
$$
for all $t\in [0,T]$.
\end{lemma}

It will be clear from the proof of Lemma \ref{lem:A} that the constant $C(T)$ growths exponentially in $T>0$.

\begin{proof}[Proof of Lemma  \ref{lem:A}.]

Since $v_\ve/(u_\ve+v_\ve)\leq 1$ for nonnegative solutions,
equation  \eqref{eq1epA} yields the inequality,
\begin{equation}\label{est_uep}
\|u_{\ve}(t)\|_{L^{\infty}(\Omega)}\leq M(t)\equiv \|u_0\|_{L^{\infty}(\Omega)} e^{(a-d_c)t}.
\end{equation}
Hence, we have the estimate $\|u_{\ve}(t)\|_{L^{\infty}(\Omega)}\leq C(T)= \|u_0\|_{L^{\infty}(\Omega)} e^{(a-d_c)T}$ for all $t\in [0,T]$.

Applying  the comparison principle to the parabolic equation  \eqref{eq3epA} with the Neumann boundary condition 
we obtain the estimate
\begin{equation}\label{ineq:w:C}
0 \leq w_{\ve} (x,t) \leq  \| w_{\ve} (t)\|_{L^{\infty}(\Omega)} \leq C_w (t),
\end{equation}
where  the function $C_w=C_w(t)$ satisfies the Cauchy problem
\begin{equation}\label{Cw_esteq}
\begin{split}
\frac{d}{dt} C_{w}+d_g C_{w} &=  d \|v_{\ve}(t)\|_{L^{\infty}(\Omega)} +\kappa_0 \\
 C_{w}(0)&= \| w_0\|_{L^{\infty}(\Omega)},
\end{split}
\end{equation}
and is given by the formula 
\begin{equation}\label{Cw_est}
C_{w}(t) = \frac{\kappa_0}{d_g} +\left( \| w_0\|_{L^{\infty}(\Omega)}- \frac{\kappa_0}{d_g} \right) e^{-d_g t} +  d \int_0^t e^{-d_g(t-\tau)}\| v_{\ve}(\tau)\|_{L^{\infty}(\Omega)} d\tau.
\end{equation}

Next, we use   equation \eqref{eq2epA} to obtain
\begin{equation}\label{est:ve}
\|v_{\ve}(t)\|_{L^{\infty}(\Omega)} \leq \|v_0\|_{L^{\infty}(\Omega)}e^{-\frac{d_b+d}{\ve}t} + \frac{1}{\ve} \int_0^t M^2(\tau)\| w_{\ve}(\tau)\|_{L^{\infty}(\Omega)} e^{-\frac{d_b+d}{\ve}(t-\tau)} d\tau.
\end{equation}
Thus, using inequality \eqref{est_uep} and plugging the above estimate into \eqref{Cw_est} yields
\begin{equation}\label{Cw_est2}
\begin{split}
C_{w}(t) \leq& \frac{\kappa_0}{d_g} +\left( \| w_0\|_{L^{\infty}(\Omega)}- \frac{\kappa_0}{d_g} \right) e^{-d_g t} + d \|v_0\|_{L^{\infty}(\Omega)}  \int_0^t e^{-d_g(t-\tau)} e^{-\frac{d_b+d}{\ve}\tau} d\tau \\
 &+ \frac{ d \|u_0\|^2_{L^{\infty}(\Omega)}}{\ve} \int_0^t e^{-d_g(t-\tau)} \int_0^{\tau}  e^{2(a-d_c)\xi} \| w_{\ve}(\xi)\|_{L^{\infty}(\Omega)} e^{-\frac{d_b+d}{\ve}(\tau-\xi)} d\xi d\tau.
\end{split}
\end{equation}
Changing the order of integration, we can simplify  the last term on the right-hand side
\begin{eqnarray*}
&& \int_0^t e^{-d_g(t-\tau)} \int_0^{\tau}  e^{2(a-d_c)\xi} \| w_{\ve}(\xi)\|_{L^{\infty}(\Omega)} e^{-\frac{d_b+d}{\ve}(\tau-\xi)} d\xi d\tau\\
&&\quad =  \int_0^t \frac{\ve  \| w_{\ve}(\xi)\|_{L^{\infty}(\Omega)} }{d_b+d- \ve d_g}e^{-d_g(t-\xi)}e^{2(a-d_c)\xi}  d\xi - \int_0^t  \frac{\ve  \| w_{\ve}(\xi)\|_{L^{\infty}(\Omega)} }{d_b+d- \ve d_g}e^{2(a-d_c)\xi}e^{\frac{d_b+d}{\ve}(\xi-t)} d\xi
\end{eqnarray*}
hence, for $0<\varepsilon < \frac{d_b+d}{d_g}$, using inequalities \eqref{ineq:w:C}
we obtain
\begin{equation}\label{w_est}
  \| w_{\ve} (t)\|_{L^{\infty}(\Omega)} \leq C + \frac{d \| u_0\|^2_{L^{\infty}(\Omega)} }{d_b+d- \ve d_g} \int_0^t e^{2(a-d_c)\xi} \|w_{\ve}(\xi)\|_{L^{\infty}(\Omega)}e^{-d_g(t-\xi)}d\xi,
  \end{equation}
  where $C=C( \| w_0\|_{L^{\infty}(\Omega)}, \| v_0\|_{L^{\infty}(\Omega)})$ is independent of $T$ and of $\ve$.
 Finally, the Gronwall inequality applied to \eqref{w_est}  implies the estimate 
\begin{equation}\label{w_estfinal}
  \| w_{\ve} (t)\|_{L^{\infty}(\Omega)} \leq C  e^{C_1(T)}\qquad \mbox{for all} \quad 0\leq t \leq T
  \end{equation}
and for 
all sufficiently small $\varepsilon>0$  
{\rm (}{\it e.g.} $\varepsilon \in \big(0, (d_b+d)/(2d_g)\big)${\rm )}, where  positive 
constants $C$ and  $C_1(T)$ are independent of $\varepsilon$.

Finally, estimate \eqref{w_estfinal} applied to inequality \eqref{est:ve}
implies an analogous bound for~$v_{\ve}$.
\end{proof}


\begin{theorem}
Let $(u_{\ve},v_{\ve},w_{\ve})$  be a solution of system \eqref{eq1epA}- \eqref{iniepA} with sufficiently small $\varepsilon>0$  
{\rm (}e.g. $\varepsilon \in \big(0, (d_b+d)/(2d_g)\big)${\rm )}
and $(u, w)$ be a solution of the corresponding  system 
\eqref{v_quasi}--\eqref{inirA}. 
For each $T>0$, there exists a constant $C(T)>0$ independent of $\varepsilon$ such that 
\begin{eqnarray}
\max_{ t\in [0,T]}\|u_{\ve}(t)-u(t)\|_{L^\infty(\Omega)} \leq C(T) \ve,\label{estuni1}\\ 
\max_{ t\in [0,T]}\|w_{\ve}(t)-w(t)\|_{L^\infty(\Omega)} \leq C(T) \ve \label{estuni2}.
\end{eqnarray}
Additionally, we also have
\begin{eqnarray}
\int_0^T\|v_{\ve}(t)-v(t)\|_{L^\infty(\Omega)} \leq C(T) \ve,\label{estuni3}
\end{eqnarray}
where $v$ is given by equation \eqref{v_quasi}.
\end{theorem}
\begin{proof}
Letting $\alpha=u_{\ve}-u$, $\beta=v_{\ve}-v$ and $\delta=w_{\ve}-w$, we obtain by the Taylor expansion the following system
\begin{align}
&\partial_t{\alpha}=f(u_{\ve},v_{\ve}) -f(u,v)= -d_c\alpha + f_1\alpha + f_2\beta, \label{eqalpha}\\
&\ve\partial_t  {\beta} = g(u_{\ve},v_{\ve},w_{\ve}) -g(u,v,w)- \ve \partial_t v =-(d+d_b)\beta +g_1\alpha + g_2 \delta-\ve  \partial_t v, \label{eqbeta}\\
&\partial_t  {\delta} -D\Delta {\delta}_{xx} +(d_g+g_2) {\delta}=  h_1 \alpha + d\beta,\label{eqdelta}
\end{align}
supplemented with the initial conditions
$$
\alpha(x,0)=0, \qquad \delta(x,0)=0, \qquad \beta(x,0)=v_0(x)-\widetilde v_0(x)
$$
with $\widetilde v_0$ obtained from $u_0$ and $w_0$ via formula \eqref{v_quasi},
and with the Neumann boundary condition for $\delta(x,t)$.
In equations \eqref{eqalpha}-\eqref{eqdelta} the following coefficients 
\begin{align*}
f_1=\frac{\partial f}{\partial u}+d_c=\frac{av^2}{(u+v)^2},\qquad 
f_2=\frac{\partial f}{\partial v}=\frac{au^2}{(u+v)^2},\\
g_1=\frac{\partial g}{\partial u}=2uw,\qquad 
g_2=\frac{\partial g}{\partial w}=-\frac{\partial h}{\partial w}=u^2,\qquad 
h_1=\frac{\partial h}{\partial u}=-2uw
\end{align*}
are calculated in certain intermediate points and are bounded independently of $\varepsilon$ due to Lemma \ref{lem:A}.

%
 The proof is divided into three steps.

{\it Step 1:} First, applying  the comparison principle to the parabolic equation \eqref{eqdelta}
with the Neumann boundary condition and with the zero initial datum 
we obtain the estimate
$$
\|\delta(\cdot,t)\|_{L^{\infty}(\Omega)}\leq C_\delta(t) \qquad\text{for every}\quad t\in [0,T],
$$
where $C_\delta$ is a solution of the Cauchy problem 
\begin{eqnarray*}
&& C_{\delta}(0)=0,\\
&&\frac{d}{dt} C_{\delta}+d_g C_{\delta} = \| h_1\|_{L^\infty(\Omega\times [0,T])}\|\alpha(t)\|_{L^{\infty}(\Omega)} +  d \|\beta(t)\|_{L^{\infty}(\Omega)}.
\end{eqnarray*}
Since
\begin{equation*}
C_{\delta}(t) =\int_0^t e^{-d_g(t-\tau)} \big(\| h_1\|_{L^{\infty}(\Omega\times [0,T])}\|\alpha(\tau)\|_{L^{\infty}(\Omega)} +  d \|\beta(\tau)\|_{L^{\infty}(\Omega)}\big)d\tau
\end{equation*}
using the Young inequality for a convolution and the estimate 
$e^{-d_g(t-\tau)}\leq 1$, we obtain
\begin{equation}\label{deltaest}
\|\delta(\cdot,t)\|_{L^{\infty}(\Omega)}\leq \|C_{\delta}\|_{L^{\infty}(0,t)}  \leq \| h_1\|_{L^{\infty}(\Omega\times [0,T])} \|\alpha\|_{L^1(0,t;L^{\infty}(\Omega))} +  d \|\beta\|_{L^1(0,t;L^{\infty}(\Omega)).}
\end{equation}

{\it Step 2:} Next, we estimate the solution $\beta=\beta(x,t)$ of equation \eqref{eqbeta}. First, note that  
$$
\|e^{-(d+d_b)\frac{\tau}{\varepsilon}}\|_{L^{1}(0,t)} \leq \frac {\varepsilon}{d+d_b}  \qquad \text{for all } \quad t>0.
$$
 The solution of equation \eqref{eqbeta} satisfies the formula 
$$
\beta(x,t) = \beta(x,0) e^{-\frac{d+d_b}{\varepsilon}t} + \int_0^t e^{-\frac{d+d_b}{\varepsilon}(t-\tau)}\left(-\partial_{\tau}v  + \frac{1}{\varepsilon} (g_1\alpha + g_2 \delta)\right)d{\tau}.
$$
Consequently, the Young inequality yields
\begin{equation}\label{betaL1}
\begin{split}
\|{\beta}&\|_{L^1(0,t;L^{\infty}(\Omega))} \\  
\leq&  \left( \|{\beta(0)}\|_{L^{\infty}(\Omega)} +   \|\partial_{\tau} v\|_{L^1(0,t;L^{\infty}(\Omega))}\right)C \varepsilon  \\
&+ \: C  \| g_1\|_{L^{\infty}(\Omega\times[0,T])} \|\alpha\|_{L^1(0,t;L^{\infty}(\Omega))} +  C  \| g_2\|_{L^{\infty}(\Omega\times[0,T])}  \|\delta\|_{L^1(0,t;L^{\infty}(\Omega))}\\
\leq& C\varepsilon +  C\left(  \| g_1\|_{L^{\infty}(\Omega\times[0,T])} +  t\| h_1\|_{L^{\infty}(\Omega\times[0,T])}  \| g_2\|_{L^{\infty}(\Omega\times[0,T])}   \right)\|{\alpha}\|_{L^1(0,t;L^{\infty}(\Omega))}\\ 
&+\:Ct d    \| g_2\|_{L^{\infty}(\Omega\times[0,T])} \|{\beta}\|_{L^1(0,t;L^{\infty}(\Omega))},
\end{split}
\end{equation}
where the last inequality results from \eqref{deltaest}.
Here, in the first inequality of \eqref{betaL1}, the function $v(x,t)$ is given by formula
\eqref{v_quasi}. 
Hence, the quantity 
$\|\partial_{\tau} v\|_{L^1(0,t;L^{\infty}(\Omega))}$ is finite
(and obviously independent of $\ve$)
for smooth solutions by equations \eqref{eqn1rA} and \eqref{eqn2rA}.

{\it Step 3:} Finally, we estimate the solution $\alpha=\alpha(x,t)$ 
of equation  \eqref{eqalpha}. Note that $\alpha(x,0)=0$ and 
\begin{equation*}
\alpha(t,x) = \int_0^t \left(f_2(\tau)\beta(\tau)   e^{-d_c (t-\tau) + \int_{\tau}^t f_1(\xi)d{\xi}}\right) d\tau.
\end{equation*}
Thus, using the Young inequality again we obtain the estimate
\begin{equation}\label{alphainfty}
\|{\alpha}\|_{L^{\infty}((0,t)\times \Omega)}   \leq  C e^{(a-d_c)T} \|f_2\|_{L^{\infty}(\Omega)} \|\beta\|_{L^1(0,t;L^{\infty}(\Omega))}
\end{equation}
as well as 
\begin{equation}\label{alphaL1}
\|{\alpha}\|_{L^1(0,t;L^{\infty}(\Omega))}   \leq  Ct \|\beta\|_{L^1(0,t;L^{\infty}(\Omega))}.
\end{equation}
Inserting inequality \eqref{alphaL1}  into \eqref{betaL1} leads to
\begin{eqnarray}\label{betaL1new}
\|{\beta}\|_{L^1(0,t;L^{\infty}(\Omega))}   \leq  C\varepsilon + C t \|{\beta}\|_{L^1(0,t;L^{\infty}(\Omega))}.
\end{eqnarray}
For $t\leq t_0=\frac{1}{2C}$ we conclude that
\begin{equation}\label{betaL1final}
\|{\beta}\|_{L^1(0,t;L^{\infty}(\Omega))}   \leq  C\varepsilon
\qquad \text{for all}\quad  t\leq t_0.
\end{equation}
Since every $t\in (t_0,T]$ can be reached after a finite number of steps, 
 estimate \eqref{betaL1final} holds for every  $t\in[0,T]$.
Furthermore,  inequality \eqref{betaL1final} applied in 
\eqref{deltaest} and 
\eqref{alphainfty}
completes the proof of  estimates \eqref{estuni1}-\eqref{estuni3}.
\end{proof}
\begin{rem}
To obtain a better estimate of $\beta$, one should
 construct an initial value layer, since $\beta|_{t=0}\not = 0$.
\end{rem}
%
%

\section{Model of early carcinogenesis -- constant steady states}
Here, we consider the space homogeneous solutions 
of the two-equation model \eqref{v_quasi}-\eqref{inirA} which satisfy 
 the corresponding kinetic system
\begin{align}
u_t &= \left(\dfrac{auw}{d_b + d + uw} - d_c \right)u, \label{eq1kA}\\
w_t &= - d_g w -\dfrac{d_b}{d_b + d}u^2w + \kappa_0,\label{eq2kA}
\end{align}
where $a$, $d_c$, $d_b$, $d$, $d_g$, $\kappa_0$ are positive
constants and we have always assumed that $a> d_c$. %
The structure of constant steady states of this system is the same as of the original three-equation model and can be characterized by the following lemma (for the proof see~\cite{MKS12}). 
\begin{lemma}
Let $\Theta = 4d_g
 \left(\dfrac{d_c}{a-d_c}\right)^2 d_b (d_b + d)$. If  $\kappa_0^2 > \Theta$ , 
then system \eqref{eq1kA}--\eqref{eq2kA} has two positive steady states $(u_- , w_-)$ and $(u_+ , w_+)$ with
\begin{equation}\label{u_stst}
u_\pm = \dfrac{d_c}{a-d_c}(d_b + d)\dfrac{1}{w_\pm} \quad\text{and}\quad 
w_\pm = \dfrac{\kappa_0 \pm \sqrt{\kappa_0^2 - \Theta}}{2d_g}. 
\end{equation}
\end{lemma}

\begin{theorem}
Let  $(u_- , w_-)$ and $(u_+ , w_+)$ be positive  steady states of system \eqref{eq1kA}--\eqref{eq2kA} given by \eqref{u_stst}. Then  $(u_+
, w_+)$ is always unstable. %
While $(u_- , w_-)$ is stable, except for the case
\begin{align*}
\dfrac{d_c}{a}(a-d_c)-d_g > 0, \quad \dfrac{\beta}{2} \le 1 \quad
 \text{and} \quad \kappa_0^2 \le \dfrac{\beta^2 \Theta}{4(\beta-1)},
\end{align*}
where 
\[
  \beta = \dfrac{\dfrac{d_c}{a}(a-d_c)}{\dfrac{d_c}{a}(a-d_c) -d_g} > 1. 
\]
\end{theorem}
\begin{proof}
Let $(\bar{u}, \bar{w})$ denote a steady state of
 \eqref{eq1kA}--\eqref{eq2kA}. %
From  direct calculations, the Jacobian matrix $J$ at $(\bar{u},
 \bar{w})$ of the nonlinear mapping defined by the right-hand side of
\eqref{eq1kA}--\eqref{eq2kA} is of the form
\[
J = 
\left(
\begin{array}{cc}
\dfrac{d_c}{a}(a-d_c) & \dfrac{(a-d_c)^2}{a(d_b + d)}\bar{u}^2 \vspace{0.2cm}\\
-2 \dfrac{d_b d_c}{a-d_c} & -d_g - \dfrac{d_b}{d_b + d} \bar{u}^2
\end{array}
\right).
\]
We know ({\it cf.} Remark \ref{rem:kin}) that 
\begin{enumerate}[(i)]
\item if 
\begin{align}
 \dfrac{d_c}{a}(a-d_c) -d_g -
 \dfrac{d_b}{d_b + d} \bar{u}^2 < 0 \quad \text{and}\quad  - \dfrac{d_g
 d_c}{a}(a-d_c) + \dfrac{d_b d_c (a-d_c)}{a(d_b + d)}\bar{u}^2 > 0, \label{eq12}
\end{align}
then all eigenvalues of $J$ have negative real parts;
\item if 
\begin{align}
 \dfrac{d_c}{a}(a-d_c) -d_g -
 \dfrac{d_b}{d_b + d} \bar{u}^2 > 0 \quad \text{or}\quad - \dfrac{d_g
 d_c}{a}(a-d_c) + \dfrac{d_b d_c (a-d_c)}{a(d_b + d)}\bar{u}^2 < 0, \label{eq13}
\end{align}
then there is an eigenvalue of $J$ which has a positive real part.
\end{enumerate}
{\it Step 1.} First, we show the stability of $(u_+, w_+)$. Using  estimates \eqref{u_stst}, the second inequality of \eqref{eq13} can be written in the form
\begin{align}
\dfrac{d_b (d_b + d)}{d_g}\left(\dfrac{a-d_c}{d_c}\right)^2 <
 \bar{w}^2. \label{eq14}
\end{align}
Note that the left-hand side of \eqref{eq14} satisfies
\[
 \dfrac{d_b (d_b + d)}{d_g}\left(\dfrac{a-d_c}{d_c}\right)^2 =
 \frac{\Theta}{4 d_g^2}. 
\]
and $w_+$ satisfies $ \left(\kappa_0 / (2d_g)\right)^2 < w_+^2 <
 \left(\kappa_0 / d_g \right)^2$. %
Therefore, it follows from the assumption $\kappa_0^2 > \Theta$ that 
\[
 \frac{\Theta}{4 d_g^2} < \frac{\kappa_0^2}{4 d_g^2} < w_+^2,
\]
which implies that the steady state $(u_+, w_+)$ is unstable. %

{\it Step 2.} Next, we show stability of $(u_- , w_-)$. %
 The second inequality of \eqref{eq12} is equivalent to $\Theta /
 (4d_g^2) > \bar{w}^2$. %
The latter inequality holds true, 
since $ w_-^2 = \frac{2\kappa_0^2 - 2\kappa_0
 \sqrt{\kappa_0^2 - \Theta} - \Theta}{4d_g^2}$ and
\begin{align*}
\Theta - \left(2\kappa_0^2 - 2\kappa_0
 \sqrt{\kappa_0^2 - \Theta} - \Theta \right) = 2\sqrt{\kappa_0^2 - \Theta}
 \left(\kappa_0 - \sqrt{\kappa_0^2 - \Theta} \right) > 0.
\end{align*}
%
Using \eqref{u_stst} and the relationship $ d_b (d_b +
 d)\left(\frac{d_c}{a-d_c}\right)^2 = \frac{\Theta}{4d_g}$, the first condition of \eqref{eq12}
 becomes
\[
 w_-^2 \left[\frac{d_c}{a} (a-d_c) - d_g \right] < \frac{\Theta}{4d_g}.
\]
If $\frac{d_c}{a} (a-d_c) - d_g \le 0$, then the above inequality always
 holds. %

Assume $\frac{d_c}{a} (a-d_c) - d_g > 0$. %
Note that $w_-^2$ is given by
\[
 w_-^2 = \dfrac{2\kappa_0^2 - 2\kappa_0 \sqrt{\kappa_0^2 - \Theta} -
 \Theta}{4d_g^2}.
\]
Therefore, it is sufficient to show the following inequality
\begin{align}
\left[2\kappa_0^2 - 2\kappa_0 \sqrt{\kappa_0^2 - \Theta} -
 \Theta \right]\left[\frac{d_c}{a} (a-d_c) - d_g \right] < d_g \Theta. \label{eq17}
\end{align}
The left-hand side of \eqref{eq17} is
\begin{align*}
&\left[2\kappa_0^2 - 2\kappa_0 \sqrt{\kappa_0^2 - \Theta} -
 \Theta \right]\left[\frac{d_c}{a} (a-d_c) - d_g \right] \\
&\quad = \dfrac{d_c}{a}(a-d_c)\left[2\kappa_0^2 - 2\kappa_0 \sqrt{\kappa_0^2 - \Theta} -
 \Theta \right] - d_g \left[2\kappa_0^2 - 2\kappa_0 \sqrt{\kappa_0^2 -
 \Theta} \right] + d_g \Theta.
\end{align*}
Hence, if 
\begin{align}
 \dfrac{d_c}{a}(a-d_c)\left[2\kappa_0^2 - 2\kappa_0 \sqrt{\kappa_0^2 - \Theta} -
 \Theta \right] - d_g \left[2\kappa_0^2 - 2\kappa_0 \sqrt{\kappa_0^2 -
 \Theta} \right] < 0, \label{eq18}
\end{align}
then inequality \eqref{eq17} holds true. %
Noting $\frac{d_c}{a}(a-d_c) - d_g > 0$, we obtain, from \eqref{eq18}, that
\begin{align}
2\left[\kappa_0^2 - \kappa_0
 \sqrt{\kappa_0^2 - \Theta} \right] <
 \dfrac{\dfrac{d_c}{a}(a-d_c)}{\dfrac{d_c}{a}(a-d_c) - d_g } \Theta =
 \beta \Theta, \label{eq20}
\end{align}
what is equivalent to 
\begin{align}
\kappa_0^2 -\dfrac{\beta}{2} \Theta <  \kappa_0 \sqrt{\kappa_0^2 - \Theta}. \label{eq21}
\end{align}
If $\beta /2 > 1$ and $\Theta < \kappa_0^2 \le (\beta /2)\Theta$, then
inequality
 \eqref{eq21} is always satisfied since the right-hand side of
 \eqref{eq21} is positive. %
The remaining cases are (i)\ $\beta /2 \le 1$ and (ii)\ $\beta / 2 > 1$ and $\kappa_0^2 > (\beta
 / 2) \Theta$. In the cases (i) and (ii), the both-sides of \eqref{eq21} are positive. %
Therefore, we calculate the square of both sides of \eqref{eq21} and obtain
\begin{align}
 \frac{\beta^2}{4 (\beta - 1)} \Theta < \kappa_0^2. \label{eq2-1204}
\end{align}
If $\beta > 2$, then $\beta / 2 > \beta^2 / (4(\beta - 1))$, while $\beta / 2 \le \beta^2 / (4(\beta - 1))$ if $\beta \le 2$. %
Therefore, the inequality \eqref{eq21} holds in case (ii). %
In case (i), \eqref{eq21} is satisfied under the condition
 \eqref{eq2-1204}. %
\end{proof}


\section*{Acknowledgments}
A.~Marciniak-Czochra was supported by European Research Council Starting Grant No 210680 ``Multiscale mathematical modelling of dynamics of structure formation in cell systems'' and Emmy Noether Programme of German Research Council (DFG). 
The work of G.~Karch was partially supported 
by the NCN grant  2013/09/B/ST1/04412.
K.~Suzuki  acknowledges JSPS the Grant-in-Aid for Scientific Research (C) 26400156.


\end{document}